\documentclass[12pt]{article}

\usepackage{amsmath,amssymb,amsfonts}
\usepackage{amsthm}
\usepackage{mathrsfs}
\usepackage{mathtools}
\usepackage{bm}
\usepackage{enumitem}
\usepackage{xcolor}
\usepackage{microtype,hyperref}
\usepackage[margin=1in]{geometry}

\usepackage[numbers,sort&compress]{natbib}

\theoremstyle{thmstyleone}
\newtheorem{theorem}{Theorem}[section]
\newtheorem{proposition}[theorem]{Proposition}
\newtheorem{lemma}[theorem]{Lemma}

\theoremstyle{thmstyletwo}
\newtheorem{remark}[theorem]{Remark}

\newtheorem{problem}[theorem]{Problem}

\numberwithin{equation}{section}

\theoremstyle{thmstylethree}

\usepackage{enumitem,float,hyperref,pgfplots,bm}
\usepackage{xcolor}
\usepgfplotslibrary{groupplots}
\pgfplotsset{compat=1.18}

\newcommand{\I}{\bm I}
\newcommand{\dd}{\,\mathrm{d}}
\newcommand{\ds}{\,\mathrm{d}s}
\newcommand{\dx}{\,\mathrm{d}x}

\newcommand{\R}{\mathbb{R}}

\newcommand{\Def}{\mathbf D}

\newcommand{\ip}[2]{\left(#1,#2\right)}
\newcommand{\dual}[2]{\left\langle #1,#2\right\rangle}

\DeclareMathOperator{\tr}{tr}
\DeclareMathOperator{\dev}{dev}
\DeclareMathOperator{\Div}{div}

\usepackage{array}
\usepackage{booktabs}
\newcolumntype{M}[1]{>{\centering\arraybackslash}m{#1}}

\usepackage{charter}
\usepackage{tikz}
\usepackage{xparse}

\begin{document}

\begin{center} 
\MakeUppercase{\bf \Large A four-field auxiliary reformulation of a} \\[.18cm]
\MakeUppercase{\bf \Large Cahn--Hilliard time step: Analysis and} \\[.18cm]
\MakeUppercase{\bf \Large conforming finite element discretization}
\end{center}

\begin{center}
{\sc Marvin Fritz} \\[.2cm]
\small Faculty of Mathematics, University of Vienna, Vienna, Austria
\end{center}

\begin{quote}\small
\textsc{Abstract.} We propose a four-field auxiliary reformulation of a time-discrete
Cahn--Hilliard step.
The construction is motivated by the scalar trace structure underlying
two-dimensional Rafetseder--Zulehner decompositions and represents the scalar
quantity $-\Delta c$ in the form $2p+\Div\bm u$ through an auxiliary scalar field
$p$ and an auxiliary vector field $\bm u$.
The resulting auxiliary problem is a mixed second-order Stokes/elasticity-type
system, while the evolution equation for the phase field retains its standard
mass-conserving gradient-flow structure.
We derive a continuous four-field
formulation that is equivalent to the classical convex-splitting mixed time step.
We also state a conforming finite element discretization and prove one-step
spatial estimates for the phase-field variables together with a stable auxiliary
block estimate containing an explicit weak-Laplacian recovery defect.
The numerical experiments verify
the expected phase-field convergence in a manufactured setting, mass
conservation, energy decay, and projected consistency of the auxiliary block.
\end{quote}

\section{Introduction}\label{sec:introduction}

The Cahn--Hilliard equation is a classical diffuse-interface model for phase
separation in binary mixtures \cite{cahn1958free}.
In its standard form, it is written as a second-order system for the phase field
$c$ and the chemical potential $\mu$.
\[
  \partial_t c = m\,\Delta\mu,
  \qquad
  \mu = -\varepsilon^2\Delta c + f'(c),
\]
where $m>0$ denotes the mobility, $\varepsilon>0$ is the interfacial parameter,
and $f$ is a bulk free-energy density.
Eliminating the chemical potential yields a fourth-order parabolic equation whose
principal part is biharmonic.
A direct weak treatment of that fourth-order equation naturally leads to
$H^2$-regularity and therefore suggests $C^1$-conforming discretizations \cite{brunk2026high}.

For biharmonic-type operators, several alternatives to conforming $C^1$ methods
are available, including discontinuous Galerkin methods, $C^0$ interior penalty
methods, and mixed formulations with auxiliary variables; see, for instance,
\cite{ciarletraviart1974,brenner2012c0,brennermonksun2015}.
Among the mixed approaches, the Rafetseder--Zulehner framework provides a
representation of the bending-moment tensor through one scalar and one vector
field, thereby reducing the problem to a sequence of second-order
subproblems; see
\cite{rafetseder2018decomposition,rafetseder2019new,kosin2023new}.
This point of view is particularly attractive because it requires only $H^1$-type
regularity for the auxiliary variables and has proved effective even in settings
where nontrivial boundary coupling must be treated explicitly
\cite{rafetseder2019new,kosin2023new}.
At the same time, it is well known that naive decouplings of biharmonic
problems may fail on nonconvex domains, so the precise scalar trace relation is a
structural issue rather than a cosmetic one \cite{zhangzhang2008invalidity}.

In the numerical analysis of the Cahn--Hilliard equation, a large literature is
available for standard finite element discretizations, mixed methods, and
energy-stable time discretizations.
Classical finite element studies go back at least to Elliott and French
\cite{elliottfrench1987}, with further analysis for logarithmic free energies in
\cite{copettielliott1992}.
Second-order splitting strategies were studied in
\cite{elliottfrenchmilner1989}, and Eyre's convex-splitting idea
\cite{eyre1998} has become one of the standard references for unconditionally
gradient-stable time marching.
Further important developments include finite element methods for degenerate
mobilities \cite{barrettbloweygarcke1999}, nonlocal diffusion
\cite{brunk2025analysis}, mixed finite element error analysis
\cite{fengprohl2004}, and discontinuous Galerkin discretizations
\cite{kaystylessuli2009}; see also the broad numerical overviews for Allen--Cahn and
Cahn--Hilliard equations \cite{shenyang2010,brunk2026review}.

The present paper develops a four-field auxiliary reformulation of one
time-discrete Cahn--Hilliard step that employs the two-dimensional
scalar trace structure associated with Rafetseder--Zulehner-type decompositions.
We retain the standard mass-conserving evolution equation and represent
the elliptic quantity $-\Delta c$ through an explicit auxiliary mixed block.
This leads to four unknowns
\(
  (c,\mu,p,\bm u)
\)
where the pair $(p,\bm u)$ describes the Laplacian contribution through the
identity
\(
  -\Delta c = 2p + \Div\bm u
\).
We exhibit a rigorous, gauge-fixed auxiliary block of Stokes/elasticity type
that explicitly represents the scalar Laplacian contribution using only
$H^1$-conforming spaces, while leaving the standard $H^{-1}$-gradient-flow
evolution equation unchanged.
In particular, the construction is compatible with the classical
convex-splitting time step without altering the underlying phase-field update.
Although the analysis below is presented for a first-order backward Euler
convex-splitting step, the auxiliary block is independent of this particular
time discretization; its extension to second-order convex-splitting variants is
discussed in Remark~\ref{rem:second_order_extensions}.

This auxiliary representation is useful at two levels.
First, it gives a mathematically explicit and fully standard mixed formulation
for the scalar higher-order contribution, avoiding abstract quotient spaces and
boundary multipliers.
Second, it provides a natural structural starting point for possible extensions
in which one may wish to postprocess the Laplacian contribution, derive
auxiliary-block-based estimators, or study more general higher-order,
anisotropic, or coupled Cahn--Hilliard-type models, such as source-driven or
poroelastic variants \cite{fritz2026structure,fritz2024well,brunk2025forch}, in
which tensorial second-derivative information plays a more central role.

The paper is organized as follows.
In Section~\ref{sec:model}, we recall the Cahn--Hilliard model, the convex-splitting
time step, and the standard mixed weak formulation.
Section~\ref{sec:fourfield} introduces the RZ-type auxiliary block for the
scalar quantity $-\Delta c$ and establishes its well-posedness.
Further, we combine this block with the classical mixed
time step and obtain an explicit four-field reformulation.
Section~\ref{sec:disc} presents the corresponding conforming finite element
discretization.
Section~\ref{sec:wellposed} proves the equivalence of the four-field system with the
classical convex-splitting time step on the natural regularity class.
Finally, Section~\ref{sec:apriori} derives one-step phase-field estimates and a
conditional auxiliary-block estimate with an explicit weak-Laplacian recovery
defect, and Section~\ref{sec:numerics} presents manufactured and transient
numerical tests.

\section{Model problem and classical time discretization}\label{sec:model}

Throughout the paper, let $\Omega\subset\R^2$ be a bounded, connected,
convex polygonal domain with boundary $\partial\Omega$ and outer unit normal $n$.
We restrict the discussion to two space dimensions because the planar
Rafetseder--Zulehner decomposition is the underlying structural motivation.

\subsection{The Cahn--Hilliard system}

We consider the Cahn--Hilliard system
\begin{equation}\label{eq:CH_system}
  \partial_t c = m\,\Delta\mu,
  \qquad
  \mu = -\varepsilon^2\Delta c + f'(c)
  \qquad \text{in }\Omega\times(0,T],
\end{equation}
where $m>0$ is a constant mobility, $\varepsilon>0$ is the interfacial
parameter, and $f$ is a bulk potential.
We impose homogeneous no-flux boundary conditions
\begin{equation}\label{eq:CH_bc}
  \partial_n c = 0,
  \qquad
  \partial_n \mu = 0
  \qquad \text{on }\partial\Omega\times(0,T].
\end{equation}
Testing the first equation in \eqref{eq:CH_system} with the constant function $1$
shows conservation of mass:
\begin{equation}\label{eq:mass_cont}
  \frac{\mathrm d}{\mathrm dt}\int_\Omega c(x,t)\,\dd x = 0.
\end{equation}

\subsection{One backward Euler step}

Let
\(
  0=t_0<t_1<\dots<t_N=T
\)
be a partition of the time interval with time step
\(
  \tau=t_{n+1}-t_n.
\)

\begin{problem}
Given $c^n\approx c(\cdot,t_n)$, the backward Euler step reads
\begin{equation}\label{eq:BE}
  \frac{c^{n+1}-c^n}{\tau} = m\,\Delta\mu^{n+1},
  \qquad
  \mu^{n+1} = -\varepsilon^2\Delta c^{n+1} + f'(c^{n+1})
  \qquad \text{in }\Omega,
\end{equation}
with
\begin{equation}\label{eq:BE_bc}
  \partial_n c^{n+1}=0,
  \qquad
  \partial_n \mu^{n+1}=0
  \qquad \text{on }\partial\Omega.
\end{equation}
\end{problem}
Again, testing the first equation with $1$ yields
\begin{equation}\label{eq:mass_disc}
  \int_\Omega c^{n+1}\,\dd x = \int_\Omega c^n\,\dd x.
\end{equation}

\subsection{Convex-splitting variant}

Assume that the potential admits a convex splitting
\(
  f = f_c - f_e
\)
where $f_c,f_e\in C^1(\R)$ are convex.
Then the standard convex-splitting time step replaces the constitutive relation
in \eqref{eq:BE} by
\begin{equation}\label{eq:CS_constitutive}
  \mu^{n+1}
  =
  -\varepsilon^2\Delta c^{n+1}
  + f_c'(c^{n+1})
  - f_e'(c^n).
\end{equation}
For compactness, we write
\begin{equation}\label{eq:Gn}
  G^n(c):=
  f_c'(c)-f_e'(c^n).
\end{equation}
Set
\(
  V:=H^1(\Omega).
\)
Given $c^n\in H^1(\Omega)$, the classical weak mixed formulation of one
convex-splitting step is:

\begin{problem}
Find $(c,\mu)\in V\times V$ such that
\begin{align}
  \ip{\frac{c-c^n}{\tau}}{v}
  + m\,\ip{\nabla\mu}{\nabla v}
  &= 0
  &&\forall\,v\in V,
  \label{eq:classical1}\\[1mm]
  \ip{\mu}{w}
  - \varepsilon^2\ip{\nabla c}{\nabla w}
  - \ip{G^n(c)}{w}
  &= 0
  &&\forall\,w\in V.
  \label{eq:classical2}
\end{align}
\end{problem}
For later use, define the mean value
\[
  \bar c^n:=\frac{1}{|\Omega|}\int_\Omega c^n\,\dd x
\]
and the affine mass-constrained space
\begin{equation}\label{eq:Vcbar}
  V_{\bar c^n}
  :=
  \left\{
    v\in H^1(\Omega):
    \frac{1}{|\Omega|}\int_\Omega v\,\dd x=\bar c^n
  \right\}.
\end{equation}

\begin{remark}[On second-order convex-splitting time discretizations]
\label{rem:second_order_extensions}
The present paper is restricted to a single first-order backward Euler
convex-splitting step in order to isolate the auxiliary reformulation of the
elliptic Cahn--Hilliard contribution. The auxiliary construction itself,
however, is not tied to the backward Euler time discretization. It only uses
that the constitutive relation contains an implicit elliptic term of the form
\(
  -\varepsilon^2\Delta c^{n+1}.
\)
For example, in a BDF2 convex-splitting scheme one may replace the first
equation by
\[
  \ip{\frac{3c^{n+1}-4c^n+c^{n-1}}{2\tau}}{v}
  +m\ip{\nabla\mu^{n+1}}{\nabla v}=0,
\]
and use a second-order extrapolated explicit part in the chemical potential,
for instance
\[
  \mu^{n+1}
  =
  -\varepsilon^2\Delta c^{n+1}
  +f_c'(c^{n+1})
  -\bigl(2f_e'(c^n)-f_e'(c^{n-1})\bigr),
\]
or another energy-stable second-order convex splitting. The four-field
replacement is then obtained by the same substitution
\[
  -\Delta c^{n+1}=2p^{n+1}+\Div\bm u^{n+1},
\]
and the auxiliary block is unchanged.
Thus the equivalence argument and the auxiliary-block well-posedness extend
formally to second-order convex-splitting, Crank--Nicolson, or BDF2 variants
whenever the corresponding standard \((c,\mu)\)-scheme is well posed and energy
stable. Extending the full error analysis to such schemes would mainly
require combining the known second-order convex-splitting estimates with the
same auxiliary recovery estimate used below. We therefore regard this as a
natural extension of the present structural reformulation. Representative second-order energy-stable or
convex-splitting Cahn--Hilliard discretizations include mixed finite element,
finite-difference, nonlocal, and pseudo-spectral variants; see, for example,
\cite{diegelwangwise2016,guowangwiseyue2016,guanlowengrubwangwise2014,
chengwangwiseyue2016,shenyang2010}.
\end{remark}

\section{The continuous four-field reformulation}\label{sec:fourfield}

In this section, we formulate the continuous four-field reformulation of the Cahn--Hilliard equation. We begin by isolating the scalar quantity $-\Delta c$ and representing it by an
explicit auxiliary mixed block.
The construction is motivated by the two-dimensional scalar trace
structure associated with Rafetseder--Zulehner-type decompositions \cite{rafetseder2018decomposition,rafetseder2019new}: a scalar
second-derivative contribution is represented through a pair $(p,\bm u)$ by a
quantity of the form $2p+\Div\bm u$.
In the Cahn--Hilliard application, we apply this structure to the scalar
source
\(
  r_c:=-\Delta c
\)
so that the represented quantity is
\(
  -\Delta c = 2p+\Div\bm u
\).

\subsection{Basic identities}

The auxiliary block introduced below provides a Stokes/elasticity-type
representation of the scalar quantity $-\Delta c$.
Its trace structure is reminiscent of the
Rafetseder--Zulehner decomposition of the Hessian in two space dimensions; we
therefore refer to it as an \emph{RZ-type} block to emphasize the structural
analogy.

Let
\[
  Q_0:=L^2_0(\Omega)
  :=
  \left\{
    q\in L^2(\Omega): \int_\Omega q\,\dd x = 0
  \right\},
  \qquad
  \bm U:=[H^1_0(\Omega)]^2.
\]
For $\bm v=(v_1,v_2)$, let
\(
  \Def\bm v
  :=
  \frac12\bigl(\nabla\bm v + (\nabla\bm v)^T\bigr)
\)
denote the symmetric gradient.
For a symmetric matrix field $\bm A$ in two dimensions, we write
\[
  \dev\bm A := \bm A-\frac12(\tr\bm A)\I
\]
for its trace-free part.
Further, we introduce the rotation matrix
\[
  H:=
  \begin{pmatrix}
    0 & -1\\
    1 & \phantom{-}0
  \end{pmatrix}.
\]
For $(p,\bm u)\in Q_0\times \bm U$ define the associated symmetric tensor
\begin{equation}\label{eq:RZ_tensor}
  \mathcal M(p,\bm u)
  :=
  p\,\I + H^T(\Def\bm u)H.
\end{equation}
Since orthogonal conjugation preserves symmetry and the trace, we have
\begin{equation}\label{eq:trace_M}
  \tr \mathcal M(p,\bm u)
  =
  2p + \Div\bm u.
\end{equation}

\subsection{The explicit Neumann-compatible auxiliary block}

Given a scalar source $r\in Q_0$, we define $(p,\bm u)$ by the mixed problem:

\begin{problem}
Find $(p,\bm u)\in Q_0\times\bm U$ such that
\begin{align}
  \ip{2p+\Div\bm u}{q}
  &= \ip{r}{q}
  &&\forall\,q\in Q_0,
  \label{eq:RZblock1}\\[1mm]
  \ip{\Def\bm u}{\Def\bm z}
  + \ip{p}{\Div\bm z}
  &= 0
  &&\forall\,\bm z\in \bm U.
  \label{eq:RZblock2}
\end{align}
\end{problem}

This is a standard mixed elasticity/Stokes-type system.
The first equation prescribes the scalar quantity $2p+\Div\bm u$, while the
second equation selects the unique vector potential fixed by the homogeneous
Dirichlet gauge.

\begin{proposition}[Well-posedness of the auxiliary block]\label{prop:RZblock}
For every $r\in Q_0$, there exists a unique pair
\(
  (p,\bm u)\in Q_0\times \bm U
\)
satisfying \eqref{eq:RZblock1}--\eqref{eq:RZblock2}.
Moreover, there exists a constant $C>0$, depending only on $\Omega$, such that
\begin{equation}\label{eq:RZblock_stab}
  \|p\|_{L^2(\Omega)} + \|\bm u\|_{H^1(\Omega)}
  \le C\,\|r\|_{L^2(\Omega)}.
\end{equation}
\end{proposition}

\begin{proof}
Since $\bm u\in [H^1_0(\Omega)]^2$, we have
\[
  \int_\Omega \Div\bm u\,\dx
  =
  \int_{\partial\Omega}\bm u\cdot n\,\ds
  =
  0.
\]
Hence $\Div\bm u\in Q_0$. The first equation in
\eqref{eq:RZblock1}--\eqref{eq:RZblock2} is therefore equivalent to
\[
  p=\frac12(r-\Div\bm u)
  \qquad\text{in }L^2(\Omega).
\]
Substituting this identity into \eqref{eq:RZblock2} gives the following problem
for $\bm u$ alone:
\begin{equation}\label{eq:reduced_u_problem}
  a(\bm u,\bm z)
  =
  -\frac12\ip{r}{\Div\bm z}
  \qquad \forall\,\bm z\in\bm U,
\end{equation}
where
\[
  a(\bm u,\bm z)
  :=
  \ip{\Def\bm u}{\Def\bm z}
  -\frac12\ip{\Div\bm u}{\Div\bm z}.
\]
Using the two-dimensional trace decomposition
\[
  |\Def\bm v|^2
  =
  |\dev\Def\bm v|^2
  +\frac12|\Div\bm v|^2,
\]
we obtain
\[
  a(\bm v,\bm v)
  =
  \|\dev\Def\bm v\|_{L^2(\Omega)}^2
  \qquad \forall\,\bm v\in\bm U.
\]
We use the following Korn inequality for the trace-free symmetric gradient on
$[H^1_0(\Omega)]^2$, compare \cite[Theorem 6.3-3]{ciarlet1994three}: there exists a constant $c_{\rm dev}>0$, depending only on
$\Omega$, such that
\begin{equation} \label{Eq:Korn}
  \|\dev\Def\bm v\|_{L^2(\Omega)}
  \ge
  c_{\rm dev}\|\bm v\|_{H^1(\Omega)}
  \qquad \forall\,\bm v\in\bm U .
\end{equation}
The zero-trace condition is essential here. In two space dimensions the kernel
of $\dev\Def$ on unconstrained $H^1$ spaces contains the conformal Killing
vector fields, so no such estimate holds on all of $[H^1(\Omega)]^2$. On
$[H^1_0(\Omega)]^2$, however, the trace condition removes this kernel, and the
trace-free Korn inequality gives the stated coercivity; see also
\cite{dain2006,geymonatsuquet1986} for further details.
Thus $a$ is coercive on $\bm U$. It is also continuous, and the right-hand side
in \eqref{eq:reduced_u_problem} satisfies
\[
  \frac12\left|\ip{r}{\Div\bm z}\right|
  \le
  C\|r\|_{L^2(\Omega)}\|\bm z\|_{H^1(\Omega)}.
\]
The Lax--Milgram lemma therefore gives a unique
$\bm u\in\bm U$ satisfying \eqref{eq:reduced_u_problem}, with
\[
  \|\bm u\|_{H^1(\Omega)}
  \le
  C\|r\|_{L^2(\Omega)}.
\]
Defining
\[
  p:=\frac12(r-\Div\bm u)
\]
then gives
\[
  \|p\|_{L^2(\Omega)}
  \le
  \frac12\|r\|_{L^2(\Omega)}
  + C\|\bm u\|_{H^1(\Omega)}
  \le
  C\|r\|_{L^2(\Omega)}.
\]
The pair $(p,\bm u)$ satisfies \eqref{eq:RZblock1}--\eqref{eq:RZblock2} by
construction.

Finally, if $r=0$, then the reduced problem gives
$a(\bm u,\bm u)=0$, hence $\dev\Def\bm u=0$. The Korn--deviator inequality implies
$\bm u=0$, and then $p=0$. This proves uniqueness and the stability estimate
\eqref{eq:RZblock_stab}. 
\end{proof}

\begin{remark}
The constant in \eqref{eq:RZblock_stab} depends on $\Omega$ through the
Korn--deviator inequality and the usual Sobolev constants.
The space $\bm U=[H^1_0(\Omega)]^2$ is a concrete gauge-fixed choice.
It removes the rigid-motion ambiguity from the classical RZ potential and is
particularly natural in the present homogeneous Neumann setting.
\end{remark}

\subsection{Representation of \texorpdfstring{$-\Delta c$}{-Delta c}}

The role of the block \eqref{eq:RZblock1}--\eqref{eq:RZblock2} in the
Cahn--Hilliard context is to represent the scalar quantity $-\Delta c$.
Whenever
\(
  c\in H^2(\Omega)\) with
  $\partial_n c = 0$ on $\partial\Omega$,
the function
\(
  r_c:=-\Delta c
\)
belongs to $Q_0$ because
\[
  \int_\Omega r_c\,\dx
  =
  -\int_\Omega \Delta c\,\dx
  =
  -\int_{\partial\Omega}\partial_n c\,\ds
  =
  0.
\]
Applying Proposition~\ref{prop:RZblock} to $r_c$ yields a unique pair
\(
  (p(c),\bm u(c))\in Q_0\times\bm U
\)
such that
\begin{align}
  \ip{2p(c)+\Div\bm u(c)}{q}
  &= \ip{-\Delta c}{q}
  &&\forall\,q\in Q_0,
  \label{eq:rc1}\\[1mm]
  \ip{\Def\bm u(c)}{\Def\bm z}
  + \ip{p(c)}{\Div\bm z}
  &= 0
  &&\forall\,\bm z\in \bm U.
  \label{eq:rc2}
\end{align}
In other words,
\begin{equation}\label{eq:laplace_repr}
  -\Delta c = 2p(c)+\Div\bm u(c)
  \qquad \text{in }L^2(\Omega).
\end{equation}
This identity is the scalar Rafetseder--Zulehner trace relation that replaces the
shortcut $-\Delta c=2p$.

\subsection{Four-field weak system}

We now combine the explicit auxiliary block with the classical mixed
Cahn--Hilliard time step.
The four-field formulation is stated on the class
\begin{equation}\label{eq:regularity_class}
  c\in H^2(\Omega),\qquad
  \partial_n c=0 \text{ on }\partial\Omega,\qquad
  \mu\in H^1(\Omega),\qquad
  p\in Q_0,\qquad
  \bm u\in \bm U.
\end{equation}
We also require the integrability condition
\begin{equation}\label{eq:G_L2_assumption}
  G^n(c)=f_c'(c)-f_e'(c^n)\in L^2(\Omega).
\end{equation}
Under \eqref{eq:regularity_class}--\eqref{eq:G_L2_assumption}, all terms in
\eqref{eq:4f1}--\eqref{eq:4f4} below are well defined. The problem reads:

\begin{problem} Given $c^n\in H^1(\Omega)$, we seek
\(
  (c,\mu,p,\bm u)\in V_{\bar c^n}\times H^1(\Omega)\times Q_0\times \bm U
\)
with $c\in H^2(\Omega)$, $\partial_n c=0$, and $G^n(c)\in L^2(\Omega)$ such that
\begin{align}
  \ip{\frac{c-c^n}{\tau}}{v}
  + m\,\ip{\nabla\mu}{\nabla v}
  &= 0
  &&\forall\,v\in H^1(\Omega),
  \label{eq:4f1}\\[1mm]
  \ip{\mu}{w}
  - \varepsilon^2 \ip{2p+\Div\bm u}{w}
  - \ip{G^n(c)}{w}
  &= 0
  &&\forall\,w\in L^2(\Omega),
  \label{eq:4f2}\\[1mm]
  \ip{2p+\Div\bm u}{q}
  &= \ip{-\Delta c}{q}
  &&\forall\,q\in Q_0,
  \label{eq:4f3}\\[1mm]
  \ip{\Def\bm u}{\Def\bm z}
  + \ip{p}{\Div\bm z}
  &= 0
  &&\forall\,\bm z\in \bm U.
  \label{eq:4f4}
\end{align}
\end{problem}

\begin{remark}
Equation \eqref{eq:4f2} is tested against $w\in L^2(\Omega)$ rather than
$w\in H^1(\Omega)$, so it expresses the constitutive identity
\[
  \mu = \varepsilon^2(2p+\Div\bm u)+G^n(c)
\]
as an identity in $L^2(\Omega)$.
This is consistent with \eqref{eq:regularity_class}--\eqref{eq:G_L2_assumption},
in which both $-\Delta c$ and $G^n(c)$ belong to $L^2(\Omega)$.
\end{remark}

\begin{remark}
Equation \eqref{eq:4f1} is the standard mass-conserving diffusion step.
Equation \eqref{eq:4f2} is the constitutive relation for the chemical potential,
but with the Laplacian contribution represented by the scalar quantity
$2p+\Div\bm u$.
Equations \eqref{eq:4f3}--\eqref{eq:4f4} are the explicit auxiliary block that
generates this scalar quantity from the phase field.
\end{remark}

\section{Conforming finite element discretization}\label{sec:disc}

We now derive the corresponding fully discrete scheme.
At the discrete level, the method uses only standard conforming $H^1$ finite
element spaces.

\subsection{Discrete spaces}

Let $\{\mathcal T_h\}_{h>0}$ be a shape-regular family of conforming
triangulations of $\Omega$.
For an integer $r\ge1$, let $\mathbb V_h^r\subset H^1(\Omega)$ denote the
standard conforming Lagrange finite element space of degree $r$.
We choose polynomial degrees
\[
  r_c\ge1,\qquad r_p\ge1,\qquad r_u\ge1,
\]
and define
\[
  V_h := \mathbb V_h^{r_c},
  \qquad
  Q_h := \mathbb V_h^{r_p}\cap L^2_0(\Omega),
  \qquad
  \bm U_h := [\mathbb V_h^{r_u}\cap H^1_0(\Omega)]^2.
\]
The phase field and chemical potential are both approximated in $V_h$.
We further set
\[
  V_{h,0}:=V_h\cap L^2_0(\Omega),
  \qquad
  V_{h,\bar c_h^n}
  :=
  \left\{
    v_h\in V_h:
    \frac1{|\Omega|}\int_\Omega v_h\,\dx=\bar c_h^n
  \right\}.
\]
For the equivalence with the standard conforming mixed Cahn--Hilliard step we
assume
\begin{equation}\label{eq:Q_contains_Vh0}
  V_{h,0}\subset Q_h.
\end{equation}
For the nested Lagrange spaces on the same triangulation used here, this
condition is satisfied whenever $r_p\ge r_c$. In particular, it holds for the
choice used in the numerical experiments, we use
$r_c=1$, $r_p=1$, and $r_u=2$.

We note that no additional discrete inf--sup condition is needed for the auxiliary block. The reason is
the coercivity of the symmetric bilinear form
\[
  \mathcal A\bigl((p,\bm u),(q,\bm z)\bigr)
  :=
  2\ip{p}{q}
  +\ip{\Div\bm u}{q}
  +\ip{p}{\Div\bm z}
  +\ip{\Def\bm u}{\Def\bm z}.
\]
Indeed,
\[
  \mathcal A\bigl((p,\bm u),(p,\bm u)\bigr)
  =
  \frac12\|2p+\Div\bm u\|_{L^2(\Omega)}^2
  +
  \|\dev\Def\bm u\|_{L^2(\Omega)}^2.
\]
The trace-free Korn inequality used in the proof of
Proposition~\ref{prop:RZblock} gives
\[
  \|\bm u\|_{H^1(\Omega)}
  \le C\|\dev\Def\bm u\|_{L^2(\Omega)}.
\]
Furthermore,
\[
  2\|p\|_{L^2(\Omega)}
  \le
  \|2p+\Div\bm u\|_{L^2(\Omega)}
  +\|\Div\bm u\|_{L^2(\Omega)}
  \le
  \|2p+\Div\bm u\|_{L^2(\Omega)}
  +C\|\bm u\|_{H^1(\Omega)}.
\]
Hence \(\mathcal A\) is coercive with respect to the product norm
\(\|p\|_{L^2(\Omega)}+\|\bm u\|_{H^1(\Omega)}\) on \(Q_0\times\bm U\), and
therefore also on \(Q_h\times\bm U_h\), with a mesh-independent constant.

\subsection{Discrete four-field step}

For compactness, we define
\(
  G_h^n(c_h):=f_c'(c_h)-f_e'(c_h^n)
\).
The fully discrete four-field step then reads:

\begin{problem} Given $c_h^n\in V_h$, find
\[
  (c_h,\mu_h,p_h,\bm u_h)
  \in
  V_{h,\bar c_h^n}\times V_h\times Q_h\times \bm U_h
\]
such that
\begin{align}
  \ip{\frac{c_h-c_h^n}{\tau}}{v_h}
  + m\,\ip{\nabla\mu_h}{\nabla v_h}
  &=0
  &&\forall\,v_h\in V_h,
  \label{eq:d4f1}\\[1mm]
  \ip{\mu_h}{w_h}
  - \varepsilon^2\ip{2p_h+\Div\bm u_h}{w_h}
  - \ip{G_h^n(c_h)}{w_h}
  &=0
  &&\forall\,w_h\in V_h,
  \label{eq:d4f2}\\[1mm]
  \ip{2p_h+\Div\bm u_h}{q_h}
  &= \ip{\nabla c_h}{\nabla q_h}
  &&\forall\,q_h\in Q_h,
  \label{eq:d4f3}\\[1mm]
  \ip{\Def\bm u_h}{\Def\bm z_h}
  + \ip{p_h}{\Div\bm z_h}
  &=0
  &&\forall\,\bm z_h\in \bm U_h.
  \label{eq:d4f4}
\end{align}
\end{problem}

\begin{remark}
The continuous auxiliary equation \eqref{eq:4f3} is written with the strong
quantity $-\Delta c$ because the scalar auxiliary space is
$Q_0=L^2_0(\Omega)$; thus the continuous test function $q\in Q_0$ need not
possess a gradient. At the discrete level, $c_h\in V_h\subset H^1(\Omega)$ is
only $H^1$-conforming, and the scalar quantity represented in $Q_h$ is therefore
the discrete weak Laplacian defined by
\[
  \ip{r_h}{q_h}=\ip{\nabla c_h}{\nabla q_h}
  \qquad \forall\,q_h\in Q_h.
\]
Equation \eqref{eq:d4f3} is precisely the corresponding Galerkin recovery,
written as
\[
  \ip{2p_h+\Div\bm u_h}{q_h}
  =
  \ip{\nabla c_h}{\nabla q_h}
  \qquad \forall\,q_h\in Q_h.
\]
We note that this is not a change of the continuous auxiliary space but the $C^0$ finite
element recovery of the scalar Laplacian contribution.
\end{remark}

\subsection{Discrete auxiliary block}

For later reference, we isolate the discrete auxiliary block.

\begin{problem}
    Given $r_h\in Q_h$, find $(p_h,\bm u_h)\in Q_h\times\bm U_h$ such that
\begin{align}
  \ip{2p_h+\Div\bm u_h}{q_h}
  &= \ip{r_h}{q_h}
  &&\forall\,q_h\in Q_h,
  \label{eq:disc_block1}\\[1mm]
  \ip{\Def\bm u_h}{\Def\bm z_h}
  + \ip{p_h}{\Div\bm z_h}
  &=0
  &&\forall\,\bm z_h\in \bm U_h.
  \label{eq:disc_block2}
\end{align}
\end{problem}

\begin{proposition}[Discrete well-posedness of the auxiliary block]
\label{prop:disc_RZblock}
For every $r_h\in Q_h$ there exists a unique
\(
  (p_h,\bm u_h)\in Q_h\times\bm U_h
\)
solving \eqref{eq:disc_block1}--\eqref{eq:disc_block2}.
Moreover, there exists a mesh-independent constant $C>0$ such that
\[
  \|p_h\|_{L^2(\Omega)}+\|\bm u_h\|_{H^1(\Omega)}
  \le C\|r_h\|_{L^2(\Omega)}.
\]
\end{proposition}

\begin{proof}
Let $\Pi_{Q_h}:L^2_0(\Omega)\to Q_h$ denote the $L^2$-orthogonal projection.
The first equation is equivalent to
\[
  p_h=\frac12\left(r_h-\Pi_{Q_h}\Div\bm u_h\right).
\]
Substitution into the second equation gives the reduced problem
\[
  a_h(\bm u_h,\bm z_h)
  =
  -\frac12\ip{r_h}{\Div\bm z_h}
  \qquad \forall\,\bm z_h\in\bm U_h,
\]
with
\[
  a_h(\bm u_h,\bm z_h)
  :=
  \ip{\Def\bm u_h}{\Def\bm z_h}
  -\frac12\ip{\Pi_{Q_h}\Div\bm u_h}{\Pi_{Q_h}\Div\bm z_h}.
\]
Since $\Pi_{Q_h}$ is an $L^2$-orthogonal projection,
\[
  a_h(\bm v_h,\bm v_h)
  \ge
  \|\Def\bm v_h\|_{L^2(\Omega)}^2
  -\frac12\|\Div\bm v_h\|_{L^2(\Omega)}^2
  =
  \|\dev\Def\bm v_h\|_{L^2(\Omega)}^2.
\]
The trace-free Korn inequality \eqref{Eq:Korn} therefore gives
\[
  a_h(\bm v_h,\bm v_h)
  \ge
  c\|\bm v_h\|_{H^1(\Omega)}^2
  \qquad \forall\,\bm v_h\in\bm U_h,
\]
with a constant independent of $h$.
The Lax--Milgram lemma yields a unique $\bm u_h\in\bm U_h$ and
\[
  \|\bm u_h\|_{H^1(\Omega)}
  \le C\|r_h\|_{L^2(\Omega)}.
\]
The definition of $p_h$ then gives
\[
  \|p_h\|_{L^2(\Omega)}
  \le C\|r_h\|_{L^2(\Omega)}.
\]
This proves existence, uniqueness, and stability.
\end{proof}

\subsection{Discrete energy interpretation}

It is useful to keep the classical convex-splitting discrete functional in mind.
Define the discrete mean-zero space
\[
  V_{h,0}
  =
  \left\{
    v_h\in V_h:\int_\Omega v_h\,\dx=0
  \right\},
\]
and let $\mathcal N_h:V_{h,0}\to V_{h,0}$ denote the discrete Neumann operator
given by
\[
  \ip{\nabla\mathcal N_h g_h}{\nabla v_h}
  =
  \ip{g_h}{v_h}
  \qquad \forall\,v_h\in V_{h,0}.
\]
The corresponding discrete $H^{-1}$ norm is
\[
  \|g_h\|_{-1,h}^2
  :=
  \ip{\nabla\mathcal N_h g_h}{\nabla\mathcal N_h g_h}.
\]
Then the discrete convex-splitting functional is
\begin{equation}\label{eq:J_h}
  \mathcal J_h(c_h)
  :=
  \frac{1}{2m\tau}\|c_h-c_h^n\|_{-1,h}^2
  + \int_\Omega
    \frac{\varepsilon^2}{2}|\nabla c_h|^2 + f_c(c_h)
    - f_e'(c_h^n)c_h
    \,\dx,
\end{equation}
defined on $V_{h,\bar c_h^n}$.


\section{Well-posedness and equivalence}\label{sec:wellposed}

We now show that the four-field formulation is equivalent to the classical
convex-splitting step on the natural regularity class.

\subsection{Classical convex-splitting step}

Define the mean-zero subspace
\[
  V_0
  :=
  \left\{
    v\in H^1(\Omega): \int_\Omega v\,\dx = 0
  \right\}.
\]
For $g\in V_0'$, let $\mathcal N g\in V_0$ be the unique solution of
\[
  \ip{\nabla \mathcal N g}{\nabla v}
  = \dual{g}{v}
  \qquad \forall\,v\in V_0.
\]
This induces the Neumann $H^{-1}$ norm
\[
  \|g\|_{-1}^2
  :=
  \ip{\nabla \mathcal N g}{\nabla \mathcal N g}.
\]

We recall the definition
\(
  f=f_c-f_e\)
with $f_c,f_e \in C^1(\R)$ convex.
Define
\[
  E_c(c)
  :=
  \int_\Omega
  \frac{\varepsilon^2}{2}|\nabla c|^2 + f_c(c)\,\dx,
\]
and
\begin{equation}\label{eq:J_cont}
  \mathcal J(c)
  :=
  \frac{1}{2m\tau}\|c-c^n\|_{-1}^2
  + E_c(c)
  - \int_\Omega f_e'(c^n)c\,\dx,
\end{equation}
on the affine space $V_{\bar c^n}$.
We also use the Neumann-trace-constrained Sobolev space
\[
  H^2_N(\Omega)
  :=
  \left\{
    v\in H^2(\Omega): \partial_n v = 0 \text{ on }\partial\Omega
  \right\}.
\]

\begin{theorem}[Classical convex-splitting step]\label{thm:classical}
Let \(c^n\in H^1(\Omega)\), and assume
that \(f_c'\) and \(f_e'\) have at most polynomial growth. Then the functional
\(\mathcal J\) defined in \eqref{eq:J_cont} is coercive, weakly lower
semicontinuous, and strictly convex on \(V_{\bar c^n}\). Consequently, there
exists a unique minimizer
\(
  c^{n+1}\in V_{\bar c^n}
\).
Defining
\[
  \mu_0^{n+1}:=-\frac1{m\tau}\,\mathcal N(c^{n+1}-c^n)\in V_0,
\]
there exists a unique constant \(\alpha^{n+1}\in\mathbb R\) such that
\[
  \mu^{n+1}:=\mu_0^{n+1}+\alpha^{n+1}\in H^1(\Omega)
\]
satisfies
\begin{equation}\label{eq:classical_constitutive_2}
  \ip{\mu^{n+1}}{w}
  - \varepsilon^2\ip{\nabla c^{n+1}}{\nabla w}
  - \ip{G^n(c^{n+1})}{w}
  = 0
  \qquad \forall\,w\in H^1(\Omega).
\end{equation}
Hence \((c^{n+1},\mu^{n+1})\) is the unique solution of
\eqref{eq:classical1}--\eqref{eq:classical2}.
\end{theorem}

\begin{proof}
We first prove coercivity. Let \(c\in V_{\bar c^n}\). Since the mean value of
\(c\) is fixed, Poincar\'e's inequality gives
\[
  \|c\|_{L^2(\Omega)}
  \le
  C\left(\|\nabla c\|_{L^2(\Omega)}+|\bar c^n|\right).
\]
Convexity of \(f_c\) implies the affine lower bound
\[
  f_c(s)\ge f_c(0)+f_c'(0)s
  \qquad \forall\,s\in\mathbb R.
\]
Therefore,
\[
  \int_\Omega f_c(c)\,\dx
  \ge
  |\Omega|f_c(0)+f_c'(0)\int_\Omega c\,\dx,
\]
and the right-hand side is a constant on \(V_{\bar c^n}\). Moreover, since
\(f_e'(c^n)\in L^2(\Omega)\), Young's inequality gives, for every
\(\delta>0\),
\[
  \left|\int_\Omega f_e'(c^n)c\,\dx\right|
  \le
  \|f_e'(c^n)\|_{L^2(\Omega)}\|c\|_{L^2(\Omega)}
  \le
  \delta\|\nabla c\|_{L^2(\Omega)}^2
  +C_\delta\left(1+\|f_e'(c^n)\|_{L^2(\Omega)}^2+|\bar c^n|^2\right).
\]
Choosing \(\delta<\varepsilon^2/2\), and using the nonnegativity of the
\(H^{-1}\)-term, we obtain
\[
  \mathcal J(c)
  \ge
  \frac{\varepsilon^2}{4}\|\nabla c\|_{L^2(\Omega)}^2
  -C.
\]
By Poincar\'e's inequality on the fixed-mass affine space, this proves
coercivity of \(\mathcal J\) on \(V_{\bar c^n}\).

Next, let \(c_k\rightharpoonup c\) weakly in \(H^1(\Omega)\), with
\(c_k\in V_{\bar c^n}\). Then \(c\in V_{\bar c^n}\). The quadratic gradient term
is weakly lower semicontinuous. The \(H^{-1}\)-term is also weakly lower
semicontinuous, since the map
\[
  H^1(\Omega)\ni v\mapsto v-c^n\in V_0'
\]
is linear and continuous and the squared norm on \(V_0'\) is weakly lower
semicontinuous. The integral functional
\[
  c\mapsto \int_\Omega f_c(c)\,\dx
\]
is weakly lower semicontinuous because \(f_c\) is convex and continuous. Finally,
the linear term
\[
  c\mapsto -\int_\Omega f_e'(c^n)c\,\dx
\]
is weakly continuous on \(H^1(\Omega)\), since \(f_e'(c^n)\in L^2(\Omega)\).
Thus \(\mathcal J\) is weakly lower semicontinuous.

The direct method now gives a minimizer \(c^{n+1}\in V_{\bar c^n}\). The
functional is strictly convex on \(V_{\bar c^n}\): the term
\[
  c\mapsto \frac{\varepsilon^2}{2}\|\nabla c\|_{L^2(\Omega)}^2
\]
is strictly convex on the fixed-mass affine space, because if two functions in
\(V_{\bar c^n}\) differ by a constant, then that constant must have zero mean
and hence must vanish. The remaining terms are convex or linear. Therefore the
minimizer is unique.

Let \(\eta\in V_0\). Differentiating
\(\mathcal J(c^{n+1}+s\eta)\) at \(s=0\) gives
\[
  \frac1{m\tau}
  \ip{\nabla\mathcal N(c^{n+1}-c^n)}{\nabla\mathcal N\eta}
  +\varepsilon^2\ip{\nabla c^{n+1}}{\nabla\eta}
  +\ip{G^n(c^{n+1})}{\eta}
  =0.
\]
Since
\[
  \ip{\nabla\mathcal N(c^{n+1}-c^n)}{\nabla\mathcal N\eta}
  =
  \ip{c^{n+1}-c^n}{\mathcal N\eta},
\]
the definition
\[
  \mu_0^{n+1}:=-\frac1{m\tau}\mathcal N(c^{n+1}-c^n)
\]
yields
\[
  \ip{\frac{c^{n+1}-c^n}{\tau}}{v}
  +m\ip{\nabla\mu_0^{n+1}}{\nabla v}=0
  \qquad \forall\,v\in H^1(\Omega),
\]
because constants do not contribute to the gradient term and
\(c^{n+1}-c^n\) has zero mean. The Euler--Lagrange identity also gives
\[
  \ip{\mu_0^{n+1}}{\eta}
  -\varepsilon^2\ip{\nabla c^{n+1}}{\nabla\eta}
  -\ip{G^n(c^{n+1})}{\eta}
  =0
  \qquad \forall\,\eta\in V_0.
\]
Choose the unique constant \(\alpha^{n+1}\in\mathbb R\) such that the same
identity holds for the constant test function \(1\), namely
\[
  \alpha^{n+1}|\Omega|
  =
  \varepsilon^2\ip{\nabla c^{n+1}}{\nabla 1}
  +\ip{G^n(c^{n+1})}{1}
  -\ip{\mu_0^{n+1}}{1}.
\]
Since \(\nabla 1=0\), this fixes the mean of \(\mu^{n+1}\) and yields
\eqref{eq:classical_constitutive_2} for every \(w\in H^1(\Omega)\). Therefore
\((c^{n+1},\mu^{n+1})\) satisfies
\eqref{eq:classical1}--\eqref{eq:classical2}.

Conversely, any solution of \eqref{eq:classical1}--\eqref{eq:classical2}
satisfies the Euler--Lagrange equation for \(\mathcal J\) on
\(V_{\bar c^n}\), and hence has the same phase field by uniqueness of the
minimizer. The chemical potential is then fixed uniquely by
\eqref{eq:classical_constitutive_2}.
\end{proof}

\subsection{Continuous four-field equivalence}

We now show that the four-field formulation is rigorously equivalent to the
classical mixed time step, provided $c$ lies in the natural regularity class
$H^2(\Omega)$ with homogeneous Neumann trace.

\begin{lemma}[Elliptic regularity of the convex-splitting step on convex polygons]
\label{lem:H2reg}
Let \(\Omega\subset\R^2\) be a bounded convex polygon, and let
\((c,\mu)\in H^1(\Omega)\times H^1(\Omega)\) be the solution of the classical
convex-splitting step \eqref{eq:classical1}--\eqref{eq:classical2} with
\(c^n\in H^1(\Omega)\).
Assume that \(f_c'\) and \(f_e'\) are continuous and of at most polynomial growth.
Then it gives
\(
  c\in H^2(\Omega)\) with \(\partial_n c=0\) on \(\partial\Omega\),
that is, \(c\in H^2_N(\Omega)\).
\end{lemma}

\begin{proof}
Since \(\Omega\subset\R^2\), the Sobolev embedding
\(H^1(\Omega)\hookrightarrow L^q(\Omega)\) holds for every finite \(q\), so
\(c,c^n\in L^q(\Omega)\) for all \(q<\infty\); together with the polynomial growth
of \(f_c'\) and \(f_e'\) this yields \(G^n(c)=f_c'(c)-f_e'(c^n)\in L^2(\Omega)\).
By \eqref{eq:classical2}, \(c\) is the weak solution of the Neumann problem
\(-\varepsilon^2\Delta c=\mu-G^n(c)\) with \(\partial_n c=0\). Thus, the right-hand
side belongs to \(L^2(\Omega)\) due to \(\mu\in H^1(\Omega)\hookrightarrow
L^2(\Omega)\). On a convex polygon the homogeneous
Neumann Laplacian is \(H^2\)-regular \cite{grisvard2011elliptic}, whence
\(c\in H^2(\Omega)\) with the Neumann trace retained.
\end{proof}

The following reformulation theorem shows that the auxiliary variables do not alter the phase-field
update itself.
They provide an explicit second-order representation of the Laplacian
contribution in the constitutive relation once the regularity
$c\in H^2_N(\Omega)$ is available.

\begin{theorem}[Continuous four-field reformulation]\label{thm:fourfield}
Assume the hypotheses of Theorem~\ref{thm:classical}. Let
\(
  (c,\mu)\in H^1(\Omega)\times H^1(\Omega)
\)
be the unique solution of the classical convex-splitting step
\eqref{eq:classical1}--\eqref{eq:classical2}. Then $G^n(c)\in L^2(\Omega)$, and
Lemma~\ref{lem:H2reg} gives
\(c\in H^2_N(\Omega)\).
Consequently, $-\Delta c\in L^2_0(\Omega)$, and there exists a unique pair
\(
  (p,\bm u)\in Q_0\times \bm U
\)
such that the quadruple $(c,\mu,p,\bm u)$ satisfies
\eqref{eq:4f1}--\eqref{eq:4f4}.

Conversely, let
\(
  (c,\mu,p,\bm u)\in
  V_{\bar c^n}\times H^1(\Omega)\times Q_0\times \bm U
\)
with
\[
  c\in H^2(\Omega),\qquad
  \partial_n c=0 \text{ on }\partial\Omega,\qquad
  G^n(c)\in L^2(\Omega),
\]
satisfy \eqref{eq:4f1}--\eqref{eq:4f4}. Then $(c,\mu)$ solves the classical
mixed step \eqref{eq:classical1}--\eqref{eq:classical2}. Therefore the
continuous four-field system is equivalent to the classical convex-splitting
time step on the regularity class
\[
  H^2_N(\Omega)\times H^1(\Omega)\times Q_0\times \bm U,
\]
subject to the integrability condition $G^n(c)\in L^2(\Omega)$.
\end{theorem}

\begin{proof}
Let $(c,\mu)$ solve the classical mixed step. By the additional growth
assumptions and Lemma~\ref{lem:H2reg}, $c\in H^2_N(\Omega)$. Hence
\[
  r:=-\Delta c\in L^2(\Omega),
  \qquad
  \int_\Omega r\,\dx=-\int_{\partial\Omega}\partial_n c\,\ds=0,
\]
so $r\in Q_0$. Proposition~\ref{prop:RZblock} gives a unique
$(p,\bm u)\in Q_0\times \bm U$ satisfying \eqref{eq:RZblock1}--\eqref{eq:RZblock2}
with right-hand side $r$. Thus \eqref{eq:4f3}--\eqref{eq:4f4} hold. Since
\eqref{eq:classical2} and the Neumann trace of $c$ imply
\[
  \mu = -\varepsilon^2\Delta c + G^n(c)
       = \varepsilon^2(2p+\Div\bm u)+G^n(c)
  \qquad\text{in }L^2(\Omega),
\]
equation \eqref{eq:4f2} follows. Equation \eqref{eq:4f1} is exactly
\eqref{eq:classical1}.

Conversely, let $(c,\mu,p,\bm u)$ satisfy \eqref{eq:4f1}--\eqref{eq:4f4} under
the stated regularity assumptions. From \eqref{eq:4f3} and the fact that both
$2p+\Div\bm u$ and $-\Delta c$ have zero mean, we have
\[
  2p+\Div\bm u = -\Delta c
  \qquad \text{in }L^2(\Omega).
\]
Inserting this identity into \eqref{eq:4f2} gives
\[
  \mu = -\varepsilon^2\Delta c + G^n(c)
  \qquad \text{in }L^2(\Omega).
\]
Therefore, for every $w\in H^1(\Omega)$,
\[
  \ip{\mu}{w}
  - \varepsilon^2\ip{\nabla c}{\nabla w}
  - \ip{G^n(c)}{w}
  = 0,
\]
where integration by parts uses $c\in H^2(\Omega)$ and $\partial_n c=0$.
Together with \eqref{eq:4f1}, this is exactly the classical mixed step. The
uniqueness of $(c,\mu)$ follows from Theorem~\ref{thm:classical}, and the
uniqueness of $(p,\bm u)$ follows from Proposition~\ref{prop:RZblock}.
\end{proof}

\subsection{Discrete well-posedness}

The same logic yields the discrete counterpart, provided the auxiliary scalar
space is large enough to test all mean-zero phase-field functions.

\begin{theorem}[Discrete four-field step]\label{thm:disc_wp}
Assume \eqref{eq:Q_contains_Vh0}. Let
\(c_h^n\in V_h\), and assume that \(f_c'\) and \(f_e'\) have at most polynomial
growth. Then the discrete functional \(\mathcal J_h\) from \eqref{eq:J_h} is
coercive and strictly convex on \(V_{h,\bar c_h^n}\). Hence there exists a unique
minimizer
\(
  c_h^{n+1}\in V_{h,\bar c_h^n}
\).
Let $\mu_h^{n+1}\in V_h$ be the unique chemical potential determined by the
discrete Euler--Lagrange equation associated with $\mathcal J_h$. Define
$r_h^{n+1}\in Q_h$ by
\begin{equation}\label{eq:disc_weakLap}
  \ip{r_h^{n+1}}{q_h}
  =
  \ip{\nabla c_h^{n+1}}{\nabla q_h}
  \qquad \forall\,q_h\in Q_h.
\end{equation}
Solving the discrete auxiliary block
\eqref{eq:disc_block1}--\eqref{eq:disc_block2} with right-hand side
$r_h^{n+1}$ gives a unique pair
\(
  (p_h^{n+1},\bm u_h^{n+1})\in Q_h\times\bm U_h
\).
Consequently, the quadruple
\[
  (c_h^{n+1},\mu_h^{n+1},p_h^{n+1},\bm u_h^{n+1})
  \in
  V_{h,\bar c_h^n}\times V_h\times Q_h\times\bm U_h
\]
satisfies \eqref{eq:d4f1}--\eqref{eq:d4f4}. Conversely, any solution of
\eqref{eq:d4f1}--\eqref{eq:d4f4} has its first two components equal to the
standard conforming convex-splitting solution. Hence the discrete four-field
step is equivalent to the standard conforming mixed step, with the auxiliary
variables uniquely determined by the discrete weak Laplacian recovery.
\end{theorem}

\begin{proof}
The proof of coercivity and strict convexity is the restriction of the argument
from Theorem~\ref{thm:classical} to the finite element affine space
\(V_{h,\bar c_h^n}\). Indeed, the fixed mass constraint gives the same
Poincar\'e control of \(\|c_h\|_{L^2}\) by \(\|\nabla c_h\|_{L^2}\), the convexity
of \(f_c\) gives the same affine lower bound, and the explicit term involving
\(f_e'(c_h^n)\) is controlled by Young's inequality. Hence
\(\mathcal J_h\) has a unique minimizer \(c_h^{n+1}\in V_{h,\bar c_h^n}\).

The minimization argument gives a unique $c_h^{n+1}\in V_{h,\bar c_h^n}$.
The Euler--Lagrange equations for $\mathcal J_h$ give the standard conforming
mixed convex-splitting equations: find
$(c_h^{n+1},\mu_h^{n+1})\in V_{h,\bar c_h^n}\times V_h$ such that
\[
  \ip{\frac{c_h^{n+1}-c_h^n}{\tau}}{v_h}
  +m\ip{\nabla\mu_h^{n+1}}{\nabla v_h}
  =0
  \qquad \forall\,v_h\in V_h,
\]
and
\[
  \ip{\mu_h^{n+1}}{w_h}
  -\varepsilon^2\ip{\nabla c_h^{n+1}}{\nabla w_h}
  -\ip{G_h^n(c_h^{n+1})}{w_h}
  =0
  \qquad \forall\,w_h\in V_h.
\]
The function $r_h^{n+1}\in Q_h$ is uniquely determined by
\eqref{eq:disc_weakLap}. Proposition~\ref{prop:disc_RZblock} gives a unique
$(p_h^{n+1},\bm u_h^{n+1})\in Q_h\times\bm U_h$ satisfying
\eqref{eq:d4f3}--\eqref{eq:d4f4}.

It remains to verify \eqref{eq:d4f2}. Let $w_h\in V_h$, and write
\[
  w_h = \overline w_h + w_h^0,
  \qquad
  \overline w_h:=\frac1{|\Omega|}\int_\Omega w_h\,\dx,
  \qquad
  w_h^0:=w_h-\overline w_h .
\]
Then $w_h^0\in V_{h,0}\subset Q_h$. Since $p_h^{n+1}\in L^2_0(\Omega)$ and
$\bm u_h^{n+1}\in[H^1_0(\Omega)]^2$,
\[
  \ip{2p_h^{n+1}+\Div\bm u_h^{n+1}}{1}=0.
\]
Moreover $\nabla \overline w_h=0$. Hence, using \eqref{eq:d4f3},
\[
  \ip{2p_h^{n+1}+\Div\bm u_h^{n+1}}{w_h}
  =
  \ip{2p_h^{n+1}+\Div\bm u_h^{n+1}}{w_h^0}
  =
  \ip{\nabla c_h^{n+1}}{\nabla w_h^0}
  =
  \ip{\nabla c_h^{n+1}}{\nabla w_h}.
\]
Substituting this identity into the standard discrete constitutive equation
gives \eqref{eq:d4f2}.

Conversely, suppose that $(c_h,\mu_h,p_h,\bm u_h)$ satisfies
\eqref{eq:d4f1}--\eqref{eq:d4f4}. The same decomposition of an arbitrary
$w_h\in V_h$ and the inclusion $V_{h,0}\subset Q_h$ imply
\[
  \ip{2p_h+\Div\bm u_h}{w_h}
  =
  \ip{\nabla c_h}{\nabla w_h}
  \qquad \forall\,w_h\in V_h.
\]
Therefore \eqref{eq:d4f2} is exactly the standard discrete constitutive
equation. Together with \eqref{eq:d4f1}, this proves that $(c_h,\mu_h)$ is the
standard conforming convex-splitting solution. Uniqueness of the auxiliary pair
follows from Proposition~\ref{prop:disc_RZblock}.
\end{proof}

\section{A priori error estimates}\label{sec:apriori}

In this section we derive one-step spatial estimates for the phase-field
variables and a stable auxiliary-block estimate for the auxiliary variables.
The auxiliary estimate contains an explicit weak-Laplacian recovery defect because the discrete recovery of
$-\Delta c$ is only weakly defined by the finite element phase field.
The Rafetseder--Zulehner connection enters through the scalar
trace structure $2p+\Div\bm u$ underlying the auxiliary representation.
Throughout this section, we use the discrete spaces introduced in
Section~\ref{sec:disc}:
\[
  V_h=\mathbb V_h^{r_c},
  \qquad
  Q_h=\mathbb V_h^{r_p}\cap L^2_0(\Omega),
  \qquad
  \bm U_h=[\mathbb V_h^{r_u}\cap H^1_0(\Omega)]^2.
\]
The phase field and chemical potential are both approximated in \(V_h\).

\subsection{Approximation and nonlinear assumptions}

We assume the standard approximation properties
\begin{align}
  \inf_{v_h\in V_h}\|v-v_h\|_{H^1(\Omega)}
  &\le C h^{r_c}\|v\|_{H^{r_c+1}(\Omega)},
  \label{eq:approx_V}\\
  \inf_{q_h\in Q_h}\|q-q_h\|_{L^2(\Omega)}
  &\le C h^{r_p+1}\|q\|_{H^{r_p+1}(\Omega)},
  \label{eq:approx_Q}\\
  \inf_{\bm z_h\in \bm U_h}\|\bm z-\bm z_h\|_{H^1(\Omega)}
  &\le C h^{r_u}\|\bm z\|_{H^{r_u+1}(\Omega)}.
  \label{eq:approx_U}
\end{align}
For the nonlinear terms we use a local hypothesis. Let $I\subset\R$ be an
interval containing the essential ranges of
$c$, $c_h$, $c^n$ and $c_h^n$.
Assume that $f_c$ is convex and that $f_c'$ and $f_e'$ are Lipschitz on $I$.
Thus there is a constant $L_I>0$ such that
\begin{equation}\label{eq:local_Lip_fc_fe}
  |f_c'(r)-f_c'(s)|+|f_e'(r)-f_e'(s)|
  \le L_I |r-s|
  \qquad \forall\,r,s\in I .
\end{equation}
The constants in the estimates below may depend on $I$, but not on $h$.

We define
\(
  V_{h,0}:=V_h\cap L^2_0(\Omega).
\)
For $g\in L^2_0(\Omega)$, introduce the discrete Neumann operator $\mathcal N_h g\in V_{h,0}$ as
\begin{equation}\label{eq:Nh}
  \ip{\nabla \mathcal N_h g}{\nabla v_h}
  =
  \ip{g}{v_h}
  \qquad \forall\,v_h\in V_{h,0}.
\end{equation}
This induces the discrete negative norm on $V_{h,0}$,
\begin{equation}\label{eq:Hminus1h}
  \|g_h\|_{-1,h}
  :=
  \|\nabla \mathcal N_h g_h\|_{L^2(\Omega)}
  \qquad (g_h\in V_{h,0}).
\end{equation}
For general $g\in L^2_0(\Omega)$ the same formula defines a discrete dual
seminorm. It depends only on the action of $g$ on $V_{h,0}$, equivalently on the
$L^2$-projection of $g$ onto $V_{h,0}$.

\subsection{Error estimate for the phase-field variables}

Let $(c,\mu)$ denote the exact solution of the classical convex-splitting mixed
step and let $(c_h,\mu_h)\in V_h\times V_h$ denote the corresponding conforming
finite element solution. We assume that $c_h^n\in V_h$ is a given approximation
of $c^n$ with
\[
  \int_\Omega c_h^n\,\dx=\int_\Omega c^n\,\dx.
\]

\begin{theorem}[One-step estimate for the phase-field variables]
\label{thm:phase_error}
Assume \eqref{eq:local_Lip_fc_fe}. Let
\[
  c,\mu\in H^{r_c+1}(\Omega),
  \qquad
  c\in H^2(\Omega),
  \qquad
  \partial_n c=0 \text{ on }\partial\Omega.
\]
Then there exists a constant $C>0$, independent of $h$, but possibly depending
on $\Omega$, $m$, $\tau$, $\varepsilon$, and the local Lipschitz interval $I$,
such that
\begin{equation}
\begin{aligned}
  \|c&-c_h\|_{-1,h}
  + \|\nabla(c-c_h)\|_{L^2(\Omega)}
  + \|\nabla(\mu-\mu_h)\|_{L^2(\Omega)}
\\
  &\le
  C\Bigl(
    h^{r_c}\|c\|_{H^{r_c+1}(\Omega)}
    + h^{r_c}\|\mu\|_{H^{r_c+1}(\Omega)}
    + \|c^n-c_h^n\|_{-1,h}
    + \|c^n-c_h^n\|_{L^2(\Omega)}
  \Bigr).
\end{aligned}
  \label{eq:phase_est}
  \end{equation}
Moreover, since $c-c_h$ has zero mean, the same right-hand side controls
$\|c-c_h\|_{H^1(\Omega)}$.
\end{theorem}

\begin{proof}
Let $R_h:H^1(\Omega)\to V_h$ be the mean-preserving Ritz projection associated
with
\[
  a(v,w):=\ip{\nabla v}{\nabla w}+\ip{v}{w}.
\]
Thus
\[
  \int_\Omega (R_hv-v)\,\dx=0,
  \qquad
  \|v-R_hv\|_{H^1(\Omega)}
  \le Ch^{r_c}\|v\|_{H^{r_c+1}(\Omega)}.
\]
Write
\[
  c-c_h=\rho_c+\theta_c,
  \qquad
  \rho_c:=c-R_hc,
  \qquad
  \theta_c:=R_hc-c_h,
\]
and
\[
  \mu-\mu_h=\rho_\mu+\theta_\mu,
  \qquad
  \rho_\mu:=\mu-R_h\mu,
  \qquad
  \theta_\mu:=R_h\mu-\mu_h.
\]
Since both $c$ and $c_h$ have the same mass and $R_h$ preserves the mean,
$\theta_c\in V_{h,0}$.

Subtracting the discrete mixed equations from the continuous ones gives, for all
$v_h,w_h\in V_h$,
\begin{align}
  \ip{\frac{\theta_c}{\tau}}{v_h}
  +m\ip{\nabla\theta_\mu}{\nabla v_h}
  &=
  -\ip{\frac{\rho_c-(c^n-c_h^n)}{\tau}}{v_h}
  -m\ip{\nabla\rho_\mu}{\nabla v_h},
  \label{eq:phase_error_1}\\
  \ip{\theta_\mu}{w_h}
  -\varepsilon^2\ip{\nabla\theta_c}{\nabla w_h}
  &-\ip{G^n(c)-G_h^n(c_h)}{w_h}
  =
  -\ip{\rho_\mu}{w_h}
  +\varepsilon^2\ip{\nabla\rho_c}{\nabla w_h}.
  \label{eq:phase_error_2}
\end{align}
Testing \eqref{eq:phase_error_1} with
$v_h=\mathcal N_h\theta_c\in V_{h,0}$ and \eqref{eq:phase_error_2} with
$w_h=\theta_c$ gives
\begin{align}
  \frac1\tau\|\theta_c\|_{-1,h}^2
  +m\varepsilon^2\|\nabla\theta_c\|_{L^2(\Omega)}^2
  +m\ip{G^n(c)-G_h^n(c_h)}{\theta_c}
  =
  R_1(\mathcal N_h\theta_c)-mR_2(\theta_c),
  \label{eq:phase_energy_ineq}
\end{align}
where $R_1$ and $R_2$ denote the right-hand sides of
\eqref{eq:phase_error_1} and \eqref{eq:phase_error_2}, respectively. The
identity used to cancel the coupling term is
\(
  \ip{\nabla\theta_\mu}{\nabla\mathcal N_h\theta_c}
  =\ip{\theta_\mu}{\theta_c}
\)
which follows by applying \eqref{eq:Nh} to the mean-zero part of $\theta_\mu$.

The nonlinear term is decomposed as
\[
  G^n(c)-G_h^n(c_h)
  =
  \bigl(f_c'(c)-f_c'(c_h)\bigr)
  -
  \bigl(f_e'(c^n)-f_e'(c_h^n)\bigr).
\]
Since $f_c$ is convex,
\[
  \ip{f_c'(c)-f_c'(c_h)}{c-c_h}\ge0.
\]
Using $\theta_c=(c-c_h)-\rho_c$, the Lipschitz bound
\eqref{eq:local_Lip_fc_fe}, Poincar\'e's inequality on $V_{h,0}$, and Young's
inequality, we obtain, for every $\delta>0$,
\[
\begin{aligned}
  &\left|
    \ip{f_c'(c)-f_c'(c_h)}{\rho_c}
  \right|
  +
  \left|
    \ip{f_e'(c^n)-f_e'(c_h^n)}{\theta_c}
  \right|
  \\ &\quad\le
  \delta \|\nabla\theta_c\|_{L^2(\Omega)}^2
  +
  C_\delta\left(
    \|\rho_c\|_{H^1(\Omega)}^2
    +
    \|c^n-c_h^n\|_{L^2(\Omega)}^2
  \right).
  \end{aligned}
\]
The residual terms are bounded by
\[
  |R_1(\mathcal N_h\theta_c)|
  \le
  C\left(
    \|\rho_c\|_{-1,h}
    +\|c^n-c_h^n\|_{-1,h}
    +\|\rho_\mu\|_{H^1(\Omega)}
  \right)\|\theta_c\|_{-1,h},
\]
and
\[
  |R_2(\theta_c)|
  \le
  C\left(
    \|\rho_\mu\|_{L^2(\Omega)}
    +\|\rho_c\|_{H^1(\Omega)}
  \right)\|\nabla\theta_c\|_{L^2(\Omega)}.
\]
Choosing $\delta>0$ sufficiently small in \eqref{eq:phase_energy_ineq} and using
Young's inequality yields
\[
  \|\theta_c\|_{-1,h}
  +\|\nabla\theta_c\|_{L^2(\Omega)}
  \le
  C\Bigl(
    \|\rho_c\|_{H^1(\Omega)}
    +\|\rho_\mu\|_{H^1(\Omega)}
    +\|c^n-c_h^n\|_{-1,h}
    +\|c^n-c_h^n\|_{L^2(\Omega)}
  \Bigr).
\]
The Ritz approximation estimates give the corresponding estimates for
$c-c_h$.

It remains to estimate the gradient of the chemical-potential error. Let
$\Pi_0\theta_\mu\in V_{h,0}$ denote the mean-zero part of $\theta_\mu$. Since
$\nabla\Pi_0\theta_\mu=\nabla\theta_\mu$, testing \eqref{eq:phase_error_1} with
$v_h=\Pi_0\theta_\mu$ gives
\[
  m\|\nabla\theta_\mu\|_{L^2(\Omega)}^2
  \le
  C\left(
    \|\theta_c\|_{-1,h}
    +\|\rho_c\|_{-1,h}
    +\|c^n-c_h^n\|_{-1,h}
    +\|\rho_\mu\|_{H^1(\Omega)}
  \right)
  \|\nabla\theta_\mu\|_{L^2(\Omega)}.
\]
The already established estimate for $\theta_c$ and the Ritz approximation
bounds yield the estimate for $\|\nabla(\mu-\mu_h)\|_{L^2(\Omega)}$.
\end{proof}

\begin{lemma}[Mean value of the chemical-potential error]
\label{lem:mu_mean_error}
Assume the hypotheses of Theorem~\ref{thm:phase_error}. Then
\[
  \left|
    \frac1{|\Omega|}\int_\Omega(\mu-\mu_h)\,\dx
  \right|
  \le
  C\left(
    \|c-c_h\|_{L^2(\Omega)}
    +
    \|c^n-c_h^n\|_{L^2(\Omega)}
  \right),
\]
where $C$ is independent of $h$. Consequently,
\[
  \|\mu-\mu_h\|_{H^1(\Omega)}
  \le
  C\left(
    \|\nabla(\mu-\mu_h)\|_{L^2(\Omega)}
    +
    \|c-c_h\|_{L^2(\Omega)}
    +
    \|c^n-c_h^n\|_{L^2(\Omega)}
  \right).
\]
\end{lemma}

\begin{proof}
Testing the continuous constitutive equation with $1$ gives
\[
  \int_\Omega \mu\,\dx
  =
  \int_\Omega G^n(c)\,\dx,
\]
because $\ip{\nabla c}{\nabla 1}=0$. Testing the discrete constitutive equation
\eqref{eq:d4f2} with $1\in V_h$ gives
\[
  \int_\Omega \mu_h\,\dx
  =
  \int_\Omega G_h^n(c_h)\,\dx,
\]
since $p_h\in Q_h\subset L^2_0(\Omega)$ and
$\bm u_h\in[H^1_0(\Omega)]^2$ imply
\[
  \int_\Omega (2p_h+\Div\bm u_h)\,\dx=0.
\]
Therefore
\[
  \int_\Omega(\mu-\mu_h)\,\dx
  =
  \int_\Omega
  \left[
    f_c'(c)-f_c'(c_h)
    -
    \bigl(f_e'(c^n)-f_e'(c_h^n)\bigr)
  \right]\,\dx .
\]
The local Lipschitz assumption \eqref{eq:local_Lip_fc_fe} yields
\[
  \left|
    \int_\Omega(\mu-\mu_h)\,\dx
  \right|
  \le
  C\left(
    \|c-c_h\|_{L^2(\Omega)}
    +
    \|c^n-c_h^n\|_{L^2(\Omega)}
  \right).
\]
The final estimate follows from Poincar\'e's inequality applied to
\[
  \mu-\mu_h
  -
  \frac1{|\Omega|}\int_\Omega(\mu-\mu_h)\,\dx .
\]
\end{proof}

\subsection{Error estimate for the auxiliary block}

Let \(c\in H^2_N(\Omega)\) and set
\(
  r:=-\Delta c\in Q_0
\).
Let \((p,\bm u)\in Q_0\times\bm U\) be the continuous auxiliary pair associated
with \(r\), i.e.,
\begin{align}
  \ip{2p+\Div\bm u}{q}
  &=\ip{r}{q}
  &&\forall\,q\in Q_0,
  \label{eq:aux_cont_1}\\
  \ip{\Def\bm u}{\Def\bm z}
  +\ip{p}{\Div\bm z}
  &=0
  &&\forall\,\bm z\in\bm U.
  \label{eq:aux_cont_2}
\end{align}
Given the discrete phase field \(c_h\), define the discrete weak Laplacian
representative \(r_h\in Q_h\) by
\begin{equation}\label{eq:rh_definition_apriori}
  \ip{r_h}{q_h}
  =
  \ip{\nabla c_h}{\nabla q_h}
  \qquad \forall\,q_h\in Q_h.
\end{equation}
Let \((p_h,\bm u_h)\in Q_h\times\bm U_h\) be the discrete auxiliary pair with
right-hand side \(r_h\).
We introduce the scalar consistency defect
\begin{equation}\label{eq:eta_lap_def}
  \eta_{\Delta,h}
  :=
  \|r-r_h\|_{L^2(\Omega)}.
\end{equation}
This quantity measures how well the discrete weak Laplacian representative
recovers the exact scalar Laplacian. It is a natural quantity in the present
formulation, since the auxiliary block is stable with respect to the \(L^2\)
right-hand side \(r\).

\begin{theorem}[Auxiliary block estimate]
\label{thm:aux_error}
Let
\[
  p\in H^{r_p+1}(\Omega)\cap L^2_0(\Omega),
  \qquad
  \bm u\in [H^{r_u+1}(\Omega)]^2\cap [H^1_0(\Omega)]^2.
\]
Then there exists a constant \(C>0\), independent of \(h\), such that
\begin{equation}
\begin{aligned}
  \|p&-p_h\|_{L^2(\Omega)}
  + \|\bm u-\bm u_h\|_{H^1(\Omega)}
  \\ &\le
  C\Bigl(
    \inf_{q_h\in Q_h}\|p-q_h\|_{L^2(\Omega)}
    + \inf_{\bm z_h\in\bm U_h}\|\bm u-\bm z_h\|_{H^1(\Omega)}
    + \eta_{\Delta,h}
  \Bigr).
\end{aligned}
  \label{eq:aux_est}
\end{equation}
Consequently,
\begin{equation}
\begin{aligned}
  \|p&-p_h\|_{L^2(\Omega)}
  + \|\bm u-\bm u_h\|_{H^1(\Omega)}
  \\ &\le
  C\Bigl(
    h^{r_p+1}\|p\|_{H^{r_p+1}(\Omega)}
    + h^{r_u}\|\bm u\|_{H^{r_u+1}(\Omega)}
    + \eta_{\Delta,h}
  \Bigr).
\end{aligned}
  \label{eq:aux_rate}
\end{equation}
\end{theorem}

\begin{proof}
Let \((\widehat p_h,\widehat{\bm u}_h)\in Q_h\times\bm U_h\) denote the discrete
auxiliary solution with the projected exact right-hand side
\(
  \Pi_{Q_h} r\in Q_h
\).
By the continuous and discrete reduced coercive formulations from
Proposition~\ref{prop:RZblock} and Proposition~\ref{prop:disc_RZblock}, the
Galerkin approximation with right-hand side \(\Pi_{Q_h}r\) satisfies the
quasi-optimal estimate
\[
  \|p-\widehat p_h\|_{L^2(\Omega)}
  +\|\bm u-\widehat{\bm u}_h\|_{H^1(\Omega)}
  \le
  C\left(
    \inf_{q_h\in Q_h}\|p-q_h\|_{L^2(\Omega)}
    +
    \inf_{\bm z_h\in\bm U_h}\|\bm u-\bm z_h\|_{H^1(\Omega)}
  \right).
\]
Here we use that the auxiliary block \eqref{eq:RZblock1}--\eqref{eq:RZblock2} is a
conforming Galerkin discretization of a symmetric, coercive problem. Indeed, the
associated bilinear form
\[
  \mathcal A\bigl((p,\bm u),(q,\bm z)\bigr)
  :=
  2\ip{p}{q}+\ip{\Div\bm u}{q}+\ip{p}{\Div\bm z}+\ip{\Def\bm u}{\Def\bm z}
\]
is symmetric and satisfies
\[
  \mathcal A\bigl((p,\bm u),(p,\bm u)\bigr)
  =
  \tfrac12\|2p+\Div\bm u\|_{L^2(\Omega)}^2
  +\|\dev\Def\bm u\|_{L^2(\Omega)}^2 .
\]
The trace-free Korn inequality used in Proposition~\ref{prop:RZblock} gives
\[
  \|\bm u\|_{H^1(\Omega)}
  \le C\|\dev\Def\bm u\|_{L^2(\Omega)}.
\]
Moreover,
\[
  2\|p\|_{L^2(\Omega)}
  \le
  \|2p+\Div\bm u\|_{L^2(\Omega)}
  +\|\Div\bm u\|_{L^2(\Omega)}
  \le
  \|2p+\Div\bm u\|_{L^2(\Omega)}
  +C\|\bm u\|_{H^1(\Omega)}.
\]
Thus \(\mathcal A\) is coercive on \(Q_0\times\bm U\) with respect to
\(\|p\|_{L^2(\Omega)}+\|\bm u\|_{H^1(\Omega)}\). Since
\(Q_h\times\bm U_h\subset Q_0\times\bm U\), coercivity is inherited on
\(Q_h\times\bm U_h\) with a mesh-independent constant. C\'ea's lemma then yields
the quasi-optimal estimate above.

It remains to compare
\((\widehat p_h,\widehat{\bm u}_h)\) with \((p_h,\bm u_h)\).
By the stability estimate in Proposition~\ref{prop:disc_RZblock},
\[
  \|\widehat p_h-p_h\|_{L^2(\Omega)}
  +\|\widehat{\bm u}_h-\bm u_h\|_{H^1(\Omega)}
  \le
  C\|\Pi_{Q_h}r-r_h\|_{L^2(\Omega)}.
\]
Since \(\Pi_{Q_h}\) is the \(L^2\)-orthogonal projection,
\[
  \|\Pi_{Q_h}r-r_h\|_{L^2(\Omega)}
  \le
  \|r-r_h\|_{L^2(\Omega)}
  =
  \eta_{\Delta,h}.
\]
The triangle inequality gives \eqref{eq:aux_est}, and the approximation
properties \eqref{eq:approx_Q}--\eqref{eq:approx_U} yield
\eqref{eq:aux_rate}.
\end{proof}

\begin{remark}
The estimate \eqref{eq:aux_est} separates the approximation of the
auxiliary variables from the scalar weak-Laplacian consistency defect
\(\eta_{\Delta,h}\).
This is important because the auxiliary block is stable with respect to the
\(L^2\)-right-hand side \(r=-\Delta c\), whereas \(r_h\) is obtained from the
discrete weak relation \eqref{eq:rh_definition_apriori}.
In particular, the estimate shows the consistency issue inside a
stronger but generally unavailable bound purely in terms of
\(\|\nabla(c-c_h)\|_{L^2}\).
\end{remark}

\subsection{Combined one-step estimate}

Combining the standard phase-field estimate with the auxiliary-block estimate
gives the following one-step bound for the full four-field discretization. The
estimate is unconditional for the phase-field variables in the usual
convex-splitting sense, under the stated regularity assumptions, but conditional
for the auxiliary variables through the explicit scalar recovery defect
\(\eta_{\Delta,h}\).

\begin{theorem}[Conditional four-field one-step estimate]
\label{thm:full_apriori}
Assume the hypotheses of Theorems~\ref{thm:phase_error} and
\ref{thm:aux_error}. Then there exists a constant $C>0$, independent of $h$, but
possibly depending on the constants appearing in Theorem~\ref{thm:phase_error},
such that
\begin{equation}
\begin{aligned}
  \|c-c_h\|_{-1,h}
  &+ \|c-c_h\|_{H^1(\Omega)}
  + \|\mu-\mu_h\|_{H^1(\Omega)}
  + \|p-p_h\|_{L^2(\Omega)}
  + \|\bm u-\bm u_h\|_{H^1(\Omega)}
\\
  \le
  C\Bigl(
    &h^{r_c}\|c\|_{H^{r_c+1}(\Omega)}
    + h^{r_c}\|\mu\|_{H^{r_c+1}(\Omega)}
    + h^{r_p+1}\|p\|_{H^{r_p+1}(\Omega)}
 \\
  &+ h^{r_u}\|\bm u\|_{H^{r_u+1}(\Omega)}+ \|c^n-c_h^n\|_{-1,h}
    + \|c^n-c_h^n\|_{L^2(\Omega)}
    + \eta_{\Delta,h}
  \Bigr),
\end{aligned}
  \label{eq:full_apriori}
\end{equation}
where
\(
  \eta_{\Delta,h}
  =
  \|-\Delta c-r_h\|_{L^2(\Omega)}
\)
and $r_h\in Q_h$ is defined by \eqref{eq:rh_definition_apriori}.
\end{theorem}

\begin{proof}
The estimate follows by combining \eqref{eq:phase_est} with
\eqref{eq:aux_rate}. The $H^1$-bound for $c-c_h$ follows from the gradient
estimate and the fact that $c-c_h$ has zero mean. The full $H^1$-bound for
$\mu-\mu_h$ follows from Theorem~\ref{thm:phase_error},
Lemma~\ref{lem:mu_mean_error}, and the already obtained $H^1$-bound for
$c-c_h$.
\end{proof}

\begin{remark}[Relation to the standard convex-splitting estimate and the recovery defect]
\label{rem:standard_vs_auxiliary_error}
The conditional character of Theorem~\ref{thm:full_apriori} concerns only the
auxiliary variables. Indeed, by Theorem~\ref{thm:disc_wp}, the first two
components \((c_h,\mu_h)\) of the four-field method coincide with the standard
conforming finite element solution of the convex-splitting mixed
Cahn--Hilliard step. Consequently, the estimate for the phase-field variables
\[
  \|c-c_h\|_{-1,h}
  +\|c-c_h\|_{H^1(\Omega)}
  +\|\mu-\mu_h\|_{H^1(\Omega)}
\]
is precisely the standard one-step conforming estimate for the underlying
convex-splitting scheme, under the regularity and local Lipschitz assumptions
stated in Theorem~\ref{thm:phase_error}. No additional consistency condition
from the auxiliary block is needed for these variables.

The additional term
\[
  \eta_{\Delta,h}
  =
  \|-\Delta c-r_h\|_{L^2(\Omega)}
\]
enters only when estimating the auxiliary pair \((p_h,\bm u_h)\). This is
because the auxiliary block is stable with respect to its scalar \(L^2\)
right-hand side, whereas in the actual finite element method this right-hand
side is not the exact quantity \(-\Delta c\), but the weakly recovered
representative \(r_h\in Q_h\) defined by
\eqref{eq:rh_definition_apriori}. Indeed,
\[
  \ip{r_h-\Pi_{Q_h}(-\Delta c)}{q_h}
  =
  \ip{\nabla(c_h-c)}{\nabla q_h}
  \qquad \forall\,q_h\in Q_h,
\]
and an inverse estimate gives
\[
  \eta_{\Delta,h}
  \le
  C\left(
    h^{r_p+1}\|-\Delta c\|_{H^{r_p+1}(\Omega)}
    +
    h^{-1}\|\nabla(c-c_h)\|_{L^2(\Omega)}
  \right).
\]
Therefore \(\eta_{\Delta,h}\) need not yield a closed positive convergence rate
for the actual weak recovery when \(r_c=r_p=1\). The theorem should therefore be
read as a standard phase-field estimate combined with a stable auxiliary-block
estimate containing an explicit scalar recovery defect. This is consistent with the equivalence result:
the auxiliary variables do not alter the phase-field update, but their own
\(L^2\)-based error estimate necessarily contains the scalar recovery defect.
\end{remark}

\section{Numerical experiments}\label{sec:numerics}

This section contains two numerical tests for the four-field
formulation.
First, we perform a manufactured one-step convergence study in order to verify
the spatial behavior of the discretization. The auxiliary part of this test is
a forced approximation study and should be distinguished from the conditional
estimate of Section~\ref{sec:apriori}.
Second, we consider a transient spinodal-decomposition-type experiment whose
purpose is to illustrate the qualitative behavior of the method and the
projected consistency of the auxiliary recovery.
All computations use the concrete finite element choices $r_c=1$, $r_p=1$ and $r_u=2$,
that is,
\[
  c_h,\mu_h\in V_h^1,\qquad
  p_h\in Q_h^1,\qquad
  \bm u_h\in [V_h^2\cap H^1_0(\Omega)]^2.
\]
The implementation in \cite{rathgeber2016firedrake} uses a
direct LU factorization for both the Cahn--Hilliard solve and the auxiliary
recovery block.

\subsection{Manufactured one-step convergence study}

We consider the unit square
\(
  \Omega=(0,1)^2
\)
and choose smooth manufactured exact fields
$c^n$, $c$, $\mu$, $p$ and $\bm u$
such that
$\partial_n c = 0$ and $\partial_n \mu = 0$ on $\partial\Omega$,
the auxiliary scalar variable $p$ has zero mean 
and the auxiliary vector field satisfies
$\bm u=0$ on $\partial\Omega$.
The forcing terms are generated by inserting these exact fields into the weak
four-field formulation, so that the quadruple
\(
  (c,\mu,p,\bm u)
\)
solves the discrete-in-time continuous problem by construction.
The tests are carried out on a sequence of uniform triangulations with
\(
  h = 2^{-n} \), 
$n\in\{3,4,5,6\}$.
The parameters are fixed as
$\varepsilon^2 = 10^{-3}$,
  $\tau = 10^{-3}$,
  $m=1$,
and we use the convex-splitting-type constitutive term
\(
  G^n(c)=c^3+2c-3c^n
\).
For the phase field and chemical potential we report
\(
  \|c-c_h\|_{H^1(\Omega)}\) and \(
  \|\mu-\mu_h\|_{H^1(\Omega)}\),
while for the auxiliary variables we report
\(
  \|p-p_h\|_{L^2(\Omega)}\) and \(
  \|\bm u-\bm u_h\|_{H^1(\Omega)}\).

The manufactured test is a forced one-step verification of the full four-field
system. Table~\ref{tab:manufactured_convergence} reports the resulting errors and
experimental rates.
The auxiliary equations are driven by the manufactured
continuous right-hand side obtained from the exact fields, rather than only by
the phase-field error generated in an unforced time step.
Thus the auxiliary variables display the intrinsic approximation behavior of
the chosen \(Q_h\times\bm U_h\) pair.
This explains why the observed auxiliary rates are close to second order, while
the general coupled estimate \eqref{eq:full_apriori} contains the additional
scalar consistency defect \(\eta_{\Delta,h}\).

\begin{table}[htbp]
\centering
\small
\caption{Manufactured one-step convergence study for the  four-field
formulation.}
\label{tab:manufactured_convergence}
\begin{tabular}{c cc cc cc cc}
\hline
$n$
& $\|c-c_h\|_{H^1}$
& rate
& $\|\mu-\mu_h\|_{H^1}$
& rate
& $\|p-p_h\|_{L^2}$
& rate
& $\|\bm u-\bm u_h\|_{H^1}$
& rate
\\
\hline
3
  & $2.032\cdot 10^{-1}$
  & --
  & $2.480\cdot 10^{-1}$
  & --
  & $8.705\cdot 10^{-1}$
  & --
  & $8.430\cdot 10^{-1}$
  & --
\\
4
  & $1.013\cdot 10^{-1}$
  & $1.00$
  & $1.255\cdot 10^{-1}$
  & $0.98$
  & $2.489\cdot 10^{-1}$
  & $1.81$
  & $2.426\cdot 10^{-1}$
  & $1.80$
\\
5
  & $5.058\cdot 10^{-2}$
  & $1.00$
  & $6.271\cdot 10^{-2}$
  & $1.00$
  & $6.381\cdot 10^{-2}$
  & $1.96$
  & $6.216\cdot 10^{-2}$
  & $1.96$
\\
6
  & $2.528\cdot 10^{-2}$
  & $1.00$
  & $3.135\cdot 10^{-2}$
  & $1.00$
  & $1.606\cdot 10^{-2}$
  & $1.99$
  & $1.563\cdot 10^{-2}$
  & $1.99$
\\
\hline
\end{tabular}
\end{table}

In total, the observed rates are consistent with the discretization and with the forced
manufactured setup.
The phase-field variables \(c_h\) and \(\mu_h\) converge with first order in the
\(H^1\)-norm, as expected for conforming piecewise affine approximation.
The auxiliary variables show nearly second-order convergence in the reported
norms; this reflects the smooth manufactured data and the intrinsic
approximation properties of the chosen \(Q_h^1\times [V_h^2]^2\) auxiliary pair.

\begin{figure}[htbp]
\centering
\includegraphics[width=.9\textwidth,page=1]{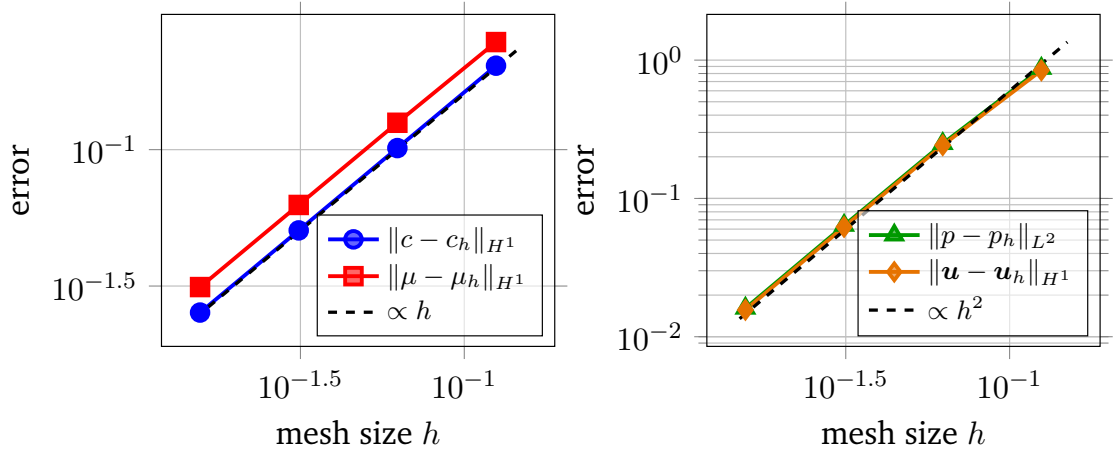}
\caption{Manufactured one-step convergence study.
Left: phase-field errors with a reference slope proportional to \(h\).
Right: auxiliary-variable errors with a reference slope proportional to \(h^2\).
The observed decay is first order for \(c_h\) and \(\mu_h\) in \(H^1\), and close
to second order for \(p_h\) in \(L^2\) and \(\bm u_h\) in \(H^1\).}
\label{fig:manufactured_convergence}
\end{figure}

\subsection{Transient experiment}

We next illustrate the four-field recovery in a transient
spinodal-decom\-po\-sition-type experiment on the unit square \(\Omega=(0,1)^2\).
The parameters are chosen as
$\varepsilon^2 = 10^{-3}$, $\tau = 10^{-3}$, $m=1$ and $T=1$
on a uniform triangulation with \(100\times100\) cells.
For the constitutive term we again use
$G^n(c_h)=c_h^3+2c_h-3c_h^n$.
The initial condition is a small random perturbation of the nearly homogeneous
state \(c\equiv -0.1\).

As seen in Figure~\ref{fig:energy_massdrift}, the experiment confirms the expected qualitative behavior.
The discrete mass is conserved up to machine precision: 
the absolute mass drift at \(T=1\) is at the level of double-precision unit
roundoff and is therefore not numerically meaningful as a finer quantity.
At the same time, the discrete free energy decreases monotonically 
in agreement with the gradient-flow structure of the convex-splitting
Cahn--Hilliard step.

\begin{figure}[htbp]
\centering
\includegraphics[width=.9\textwidth,page=2]{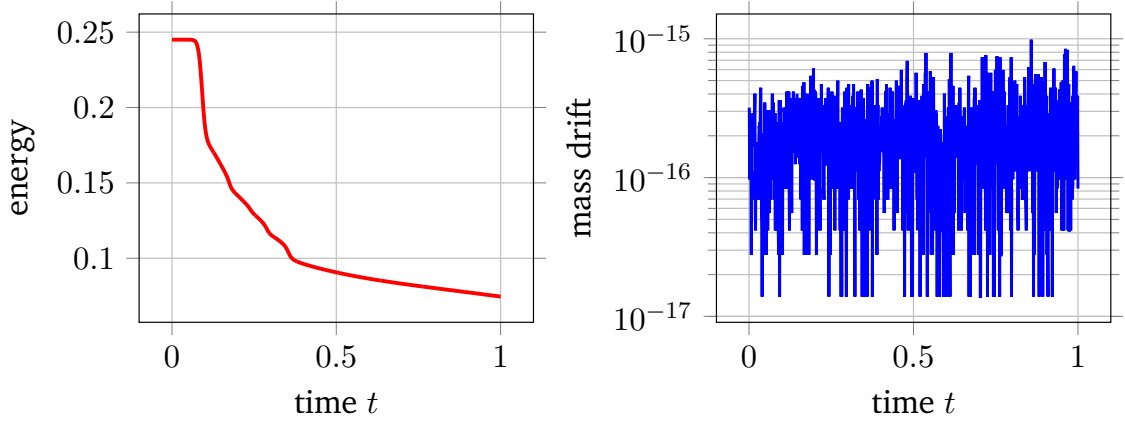}
\caption{Discrete free-energy history and absolute mass drift for the transient
four-field recovery experiment.
The free energy decreases monotonically, while the mass drift remains at the
level of machine precision throughout the computation.}
\label{fig:energy_massdrift}
\end{figure}

For the auxiliary recovery we compare the two scalar quantities in the discrete
finite element space. 
We compare the projected auxiliary quantity
\(
  \Pi_{Q_h}(2p_h+\Div\bm u_h)\in Q_h
\)
and the discrete weak Laplacian representative \(g_h\in Q_h\) (the transient
counterpart of \(r_h\) from \eqref{eq:rh_definition_apriori}), defined by
\begin{equation}\label{eq:gh_transient}
  \ip{g_h}{q_h}
  =
  \ip{\nabla c_h}{\nabla q_h}
  \qquad \forall\,q_h\in Q_h.
\end{equation}
This leads to the projected auxiliary defect
\[
  \eta_{\mathrm{aux}}
  :=
  \|\Pi_{Q_h}(2p_h+\Div\bm u_h)-g_h\|_{L^2(\Omega)}.
\]
Likewise, since the constitutive equation is enforced weakly, we monitor only
its projected mean-zero component. This distinction is important. The diffusion
equation determines the chemical potential only up to an additive constant,
whereas the full discrete constitutive equation \eqref{eq:d4f2}, tested with
constants, fixes the mean value of $\mu_h$. The quantity
$\mu_h-\varepsilon^2 g_h-\Pi_{Q_h}G^n(c_h)$ therefore contains the constant part of
$\mu_h$ and is not expected to vanish.
We therefore define the projected constitutive defect by
\[
  \eta_{\mathrm{con}}
  :=
  \|\Pi_{Q_h}\mu_h-\varepsilon^2 g_h-\Pi_{Q_h}G^n(c_h)\|_{L^2(\Omega)}.
\] 
For the
choice used in this experiment, $r_c=r_p=1$, one has $Q_h=V_{h,0}$. Hence
testing \eqref{eq:d4f2} with mean-zero functions and using \eqref{eq:d4f3} gives
\[
  \Pi_{Q_h}\mu_h
  =
  \varepsilon^2 g_h+\Pi_{Q_h}G^n(c_h)
  \qquad\text{in }Q_h,
\]
up to the algebraic solver tolerance. Thus $\eta_{\mathrm{con}}$ verifies the
projected constitutive consistency of the implementation. 
As shown in Figure~\ref{fig:defects_history}, the projected quantities remain at the level of solver tolerance throughout the
computation.

\begin{figure}[htbp]
\centering
\includegraphics[width=.9\textwidth,page=4]{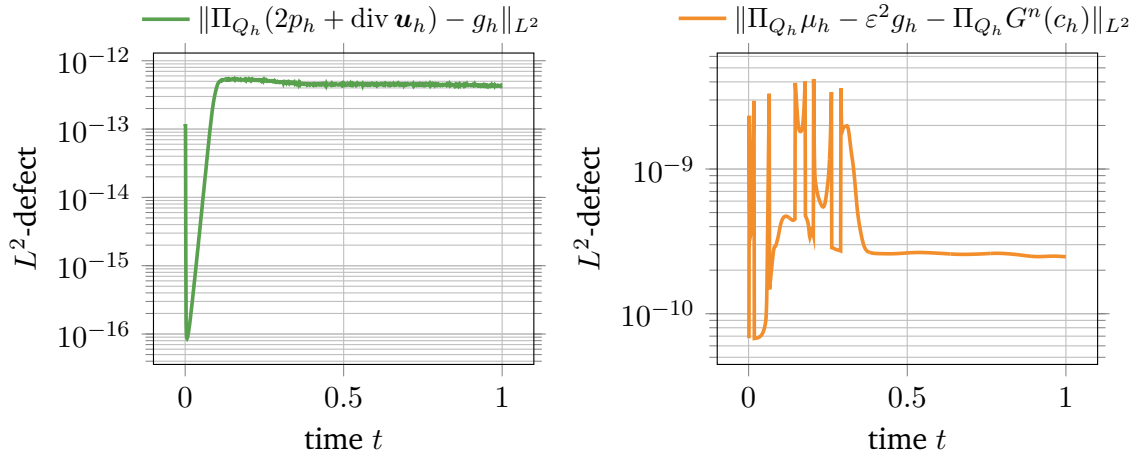}
\caption{Projected auxiliary and constitutive defects during the transient
experiment.
Both projected defects remain at the level of solver tolerance over the whole
time interval. This confirms projected algebraic consistency of the implemented
auxiliary and constitutive equations.}
\label{fig:defects_history}
\end{figure}

The snapshots in Figure~\ref{fig:fourfield-auxiliary-snapshots} illustrate the
spinodal-decomposition dynamics and the behavior of the auxiliary variables.
Starting from a small random perturbation of the nearly homogeneous state, the
phase field \(c_h\) quickly separates into positive and negative regions.  At
early times, the pattern is fine and irregular, while at later times the
interfaces coarsen and larger domains emerge.  The chemical potential \(\mu_h\)
is comparatively smoother and reflects the nonlocal driving force behind this
coarsening process.
The auxiliary variables visualize the recovered Laplacian contribution in the
constitutive relation.  The scalar field \(p_h\) develops pronounced signed
structures near the diffuse interfaces and is largest where the interfacial
curvature and the recovered second-order contribution are most pronounced.  The
auxiliary vector potential \(\bm u_h\), shown by its magnitude together with
arrows, is likewise concentrated around the evolving transition layers.  Thus
the auxiliary block does not change the phase-field evolution, but provides a
spatially resolved representation of the elliptic contribution
\(-\Delta c_h\) through the scalar trace
\(2p_h+\operatorname{div}\bm u_h\).

\begin{figure}[htbp]
\centering
\includegraphics[width=.99\textwidth,page=3]{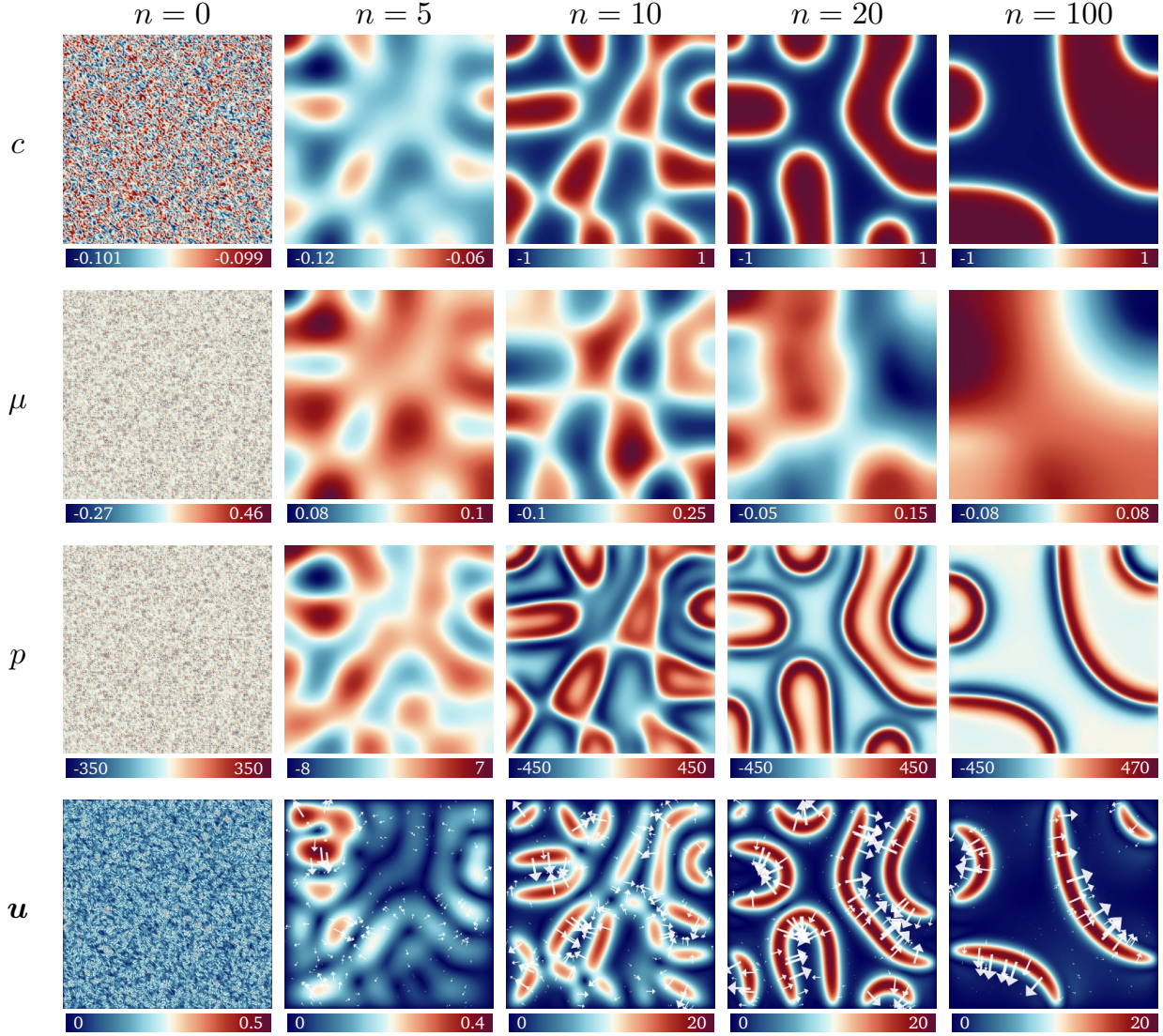}
\caption{Snapshots from the transient four-field Cahn--Hilliard experiment.
The rows show the phase field $c_h$, the chemical potential $\mu_h$, the scalar auxiliary variable $p_h$ and the
auxiliary vector potential $\bm u_h$ through its magnitude with arrows.}
\label{fig:fourfield-auxiliary-snapshots}
\end{figure}

\section{Concluding remarks}\label{sec:conclusion}

The present paper employs a
Rafetseder--Zulehner-type ansatz by using an explicit and rigorous auxiliary mixed
block for the scalar Laplacian contribution in a time-discrete Cahn--Hilliard
step.
The auxiliary problem becomes a mixed
second-order system on the concrete spaces
\(Q_0=L^2_0(\Omega)\) and
\(\bm U=[H^1_0(\Omega)]^2
\).
The resulting four-field formulation is structural in nature.
It preserves the classical mass-conserving evolution equation, represents the
scalar Laplacian contribution by an explicit auxiliary block, and uses only
standard conforming \(C^0\) finite element spaces at the discrete level.
On the natural regularity class \(c\in H^2_N(\Omega)\), the continuous
four-field system is equivalent to the classical convex-splitting mixed time
step.
At the discrete level, the inclusion $V_{h,0}\subset Q_h$ ensures equivalence
between the four-field discretization and the standard conforming mixed
convex-splitting step. The phase-field estimate is the corresponding one-step
conforming estimate, while the auxiliary variables satisfy a stable
auxiliary-block estimate containing the explicit weak-Laplacian recovery defect.
The numerical tests show the expected manufactured
convergence behavior, mass conservation, energy decay, and projected auxiliary
and constitutive consistency.

\bibliographystyle{abbrv}
\small \setlength{\bibsep}{0.4pt}
\bibliography{refs}

\end{document}